\documentclass[10pt]{amsart}
\usepackage{amsfonts, amsmath, amsthm, amssymb, amstext}
\usepackage{tikz,soul,tikz-cd,graphicx}
\usepackage{multicol,color,enumerate,relsize,fancyhdr}
\usepackage[utf8]{inputenc}
\usepackage{wasysym,mathrsfs}
\usetikzlibrary{matrix, arrows, calc}

\usepackage{amsmath,amstext,amscd, amssymb,txfonts, epsfig, psfrag, color, epstopdf}
\usepackage[normalem]{ulem}
\usepackage[colorlinks=true, pdfstartview=FitV, linkcolor=blue, citecolor=blue, urlcolor=blue]{hyperref}
\usepackage{cite}

\definecolor{cerulean}{rgb}{0,.48,.65} 
\definecolor{magenta}{rgb}{.5,0,.5} 
\definecolor{dred}{rgb}{.5,0,0} 
\definecolor{green}{rgb}{0,.5,0} 
\definecolor{blue}{rgb}{0,0,0.5} 
\definecolor{black}{rgb}{0,0,0} 
\definecolor{dgreen}{rgb}{0,.3,0} 
\definecolor{vdred}{rgb}{.3,0,0} 
\definecolor{red}{rgb}{1,0,0} 
\definecolor{salmon}{rgb}{0.98,0.50,0.45} 
\definecolor{gray}{rgb}{.5,.5,.5} 
\definecolor{seagreen}{rgb}{0.13,0.70,0.67} 
\definecolor{chartreuse}{rgb}{0.40,0.80,0.00}
\definecolor{cornflower}{rgb}{0.39,0.58,0.93} 
\definecolor{gold}{rgb}{0.80,0.68,0.00}

\renewcommand{\bar}{\overline} 

\theoremstyle{plain}
\newtheorem{theorem}{Theorem}[section]
\newtheorem{proposition}[theorem]{Proposition}
\newtheorem{corollary}[theorem]{Corollary}
\newtheorem{lemma}[theorem]{Lemma}

\theoremstyle{definition}
\newtheorem{ex}[theorem]{Example}
\newtheorem{definition}[theorem]{Definition}
\newtheorem{Examples}[theorem]{Examples}

\theoremstyle{remark}
\newtheorem{remark}[theorem]{Remark}

\newcommand{\poof}{\begin{proof}}
\newcommand{\npoof}{\end{proof}}

\newcommand{\pro}{\begin{proposition}}
\newcommand{\npro}{\end{proposition}}

\newcommand{\defn}{\begin{definition}}
\newcommand{\ndefn}{\end{definition}}

\newcommand{\thm}{\begin{theorem}}
\newcommand{\nthm}{\end{theorem}}

\newcommand{\cor}{\begin{corollary}}
\newcommand{\ncor}{\end{corollary}}

\newcommand{\bb}[1]{\expandafter\newcommand\expandafter{\csname #1\endcsname}{{\mathbb {#1}}}} 

\bb C
\bb R
\bb Q
\bb Z
\bb N

\bb F

\def\H{\hbox{\rm Hig}}

\newcommand{\Heis}{{\mathcal{H}}}
\newcommand{\Met}{{\mathcal{M}}}

\newcommand{\thus}{{\Rightarrow}}
\newcommand{\into}{\hookrightarrow}
\newcommand{\onto}{\twoheadrightarrow}

\renewcommand{\a}{\alpha}
\renewcommand{\b}{\beta}

\renewcommand{\d}{\delta}
\newcommand{\ep}{\varepsilon}

\renewcommand{\l}{\lambda}

\newcommand{\ph}{\varphi}

\newcommand{\restricted}[1]{\left|_{#1} \right.}

\newcommand{\id}{\mathrm{id}}

\newcommand{\BS}{\mathrm{BS}}
\newcommand{\SL}{\mathrm{SL}}
\newcommand{\Sym}{\mathrm{Sym}}
\renewcommand{\ker}{\mathrm{ker \:}}

\newcommand{\mrm}[1]{\expandafter\newcommand\expandafter{\csname #1\endcsname}{{\mathrm {#1}}}}

\mrm{Hom}
\mrm{Aut}
\mrm{End}

\newcommand{\set}[1]{\left\{#1\right\}}
\newcommand{\ssm}{\smallsetminus}
\newcommand{\abs}[1]{\left|#1\right|}

\newcommand{\ins}{\subseteq}

\newcommand{\by}{{\times}}

\mrm{lcm}

\begin{document}

\title[Soficity and variations on Higman's group]{Soficity and variations on Higman's group}
\author[Kassabov, Kuperberg, and Riley]{Martin Kassabov, Vivian Kuperberg, and Timothy Riley }

\date \today

\begin{abstract}
\noindent  A group is sofic when every finite subset can be well approximated in a finite symmetric group.  No example of a non-sofic group is known.  Higman's group, which is a circular amalgamation of four copies of the Baumslag--Solitar group, is a candidate.  Here we contribute to the discussion of the problem of its soficity in two ways.

We construct variations on Higman's group replacing the Baumslag--Solitar group by other groups $G$.  We give an elementary condition on $G$, enjoyed for example by $\Z \wr \Z$ and the integral Heisenberg group, under which the resulting group is sofic.

We then use soficity to deduce that there exist permutations of $\Z / n\Z$ that are seemingly pathological in that they have order dividing four and yet locally they behave like exponential functions over most of their domains.  Our approach is based on that of Helfgott and Juschenko, who recently showed the soficity of Higman's group would imply some the existence of some similarly pathological functions.  Our results call into question their suggestion that this might be a step towards proving the existence of a non-sofic group.
\end{abstract}

\maketitle

\section{Our results}
\label{results}

The word \emph{sofic}, derived from the Hebrew for \emph{finite}, was applied to a group by Weiss in~\cite{we} when every finite subset can be well approximated in a finite symmetric group or, equivalently, when the group is a subgroup of a metric ultraproduct of finite symmetric groups. The focus of this article is the outstanding open question about soficity, posed by Gromov in his 1999 paper~\cite{gr}: is every group sofic?
We will give more background on soficity in Section~\ref{soficity section}.

It is not known whether \emph{Higman's group}
$$
H_4   \ = \ \left\langle a,b,c,d \, \left| \,  b^a = b^2, c^b = c^2, d^c = d^2, a^d = a^2 \right. \right\rangle
$$
is sofic.  This group can be constructed as follows.  First amalgamate two copies of the Baumslag--Solitar group $\BS(1,2) = \langle a,b \mid b^a = b^2 \rangle$ to give
$\langle a,b,c \mid b^a = b^2, c^b = c^2 \rangle$.  By properties of the free products with amalgamation, its subgroup $\langle a, c  \rangle$ is free of rank 2.
Amalgamate with a second copy $\langle c,d,a \mid  d^c = d^2, a^d = a^2 \rangle$  along the common $\langle a, c  \rangle$ subgroup to give $H_4$.

Again, properties of free products with amalgamation tell us that the subgroups $\langle a,b \rangle$, $\langle b,c \rangle$, $\langle c,d \rangle$, and $\langle a,d \rangle$ are copies of $\BS(1,2)$, and that $\langle a,c \rangle$ is free of rank 2.  In particular, $H_4$ is not amenable, since it contains a non-abelian free subgroup.  And $H_4$ is not residually finite, because it has no finite quotients~\cite{hi}.  These properties make $H_4$ a candidate for a non-sofic group. The case is made all the more compelling because $H_4$ fails to have a property slightly more restrictive than soficity: Thom proved in~\cite{th-metric} that it does not embed into a metric ultraproduct of finite groups with a commutator-contractive invariant length function.

The building blocks for our variations on Higman's group (explored in more detail in Section~\ref{Generalizing section}) are a group $G$, subgroups $A$ and $B$, an isomorphism $\phi: B \to A$, and a $k \in \N$.  For $1 \leq i \leq k$, let $G_i$  be copies of $G$, let $A_i, B_i \leq G_i$ be copies of its subgroups $A$ and $B$, and let $\phi_{i} : B_i \to A_{i+1}$ (indices mod $k$) be the map naturally induced by $\phi$.  We define
$$
\overline{\H}_k(G, \phi)  \  := \ \langle G_1, \ldots, G_k \mid  b_i = \phi_i(b_i) \text{ for all } i \text{ and all } b_i \in B_i  \rangle,
$$
which is $k$ copies of $G$ assembled in a cyclic analog of a free product with amalgamation.
If $G = \BS(1,2) = \langle a,b \mid b^a = b^2 \rangle$ and $\phi: \langle b \rangle \to \langle a \rangle$ maps $b \mapsto a$, then $\overline \H_4(G, \phi) = H_4$.

Next we define $$ \H_k(G, \phi)   \ := \ \langle G,   t \mid  t^k = 1,  \ b^t = \phi(b);  \forall b \in B  \rangle,$$ a
semi-direct product of $\overline{\H}_k(G, \phi)$ with a cyclic group of order $k$.  The index of $\overline\H_k(G,\phi)$ in $\H_k(G,\phi)$ is $k$, so one is sofic if and only if the other is; see~\cite{pe}.

(Monod has generalized Higman's construction in a different direction in \cite{m}.)

In contrast to $H_4$, we can often prove soficity for these groups.  Indeed, in many cases they are residually solvable, and so sofic.  We will prove in Section~\ref{resfinite section}:

\begin{theorem}
\label{ressolvable}
Suppose $G$ is a residually solvable group and $\phi$ is an isomorphism $B \to A$ between subgroups $A,B \leq G$.  Suppose there exists a group homomorphism $\pi: G \to A \times B$ such that $\pi(a) = (a,1)$ for all $a \in A$ and $\pi(b) = (1,b)$ for all $b \in B$.
Then $\H_k(G, \phi)$ and $\overline{\H}_k(G, \phi)$ are residually solvable for all $k \geq 4$.
\end{theorem}

Examples of $G$ admitting such a $\pi$ include $\Z^2 = \langle a, b \mid ab=ba \rangle$, the three-dimensional integral Heisenberg group $\Heis = \langle a, b \mid [a,[a,b]] = [b,[a,b]] =1 \rangle$, $\Z \wr \Z =  \langle a \rangle \wr \langle b \rangle$, and the free metabelian group on two generators $a$ and $b$  (all with $A= \langle a \rangle$, $B= \langle b \rangle$, and $\phi: b \mapsto a$).  There is no such $\pi$ for $\BS(1,2) = \langle a, b \mid a^b=a^2 \rangle$.  See Examples~\ref{example groups} for details.

With a view to showing that $H_4$ is not sofic, Helfgott and Juschenko proved:

\begin{theorem} [Helfgott--Juschenko~\cite{he-ju}]
\label{HJ thm}
If Higman's group $H_4$ is sofic, then for all $\epsilon >0$ there exists $N \in \N$, such that for all odd $n >N$ there exists  $f \in \Sym(\Z/n\Z)$ of order dividing $4$ with $f(x+1) = 2f(x)$ for at least $(1-\ep)n$ elements $x \in \Z/n\Z$.
\end{theorem}

The $f$ of Helfgott and Juschenko's theorem behave locally like an exponential function over most of $\Z/n\Z$ but nevertheless are permutations of order dividing four.  They gave a heuristic argument as to why such $f$ are unlikely to exist, based on the assumption that these two properties are independent (an intuition that they backed up with comparisons to prominent conjectures in analytic number theory).  In Section~\ref{heuristics} we give further analysis as to why one might have expected such $f$ not to exist.

Since Helfgott and Juschenko's paper first appeared (as a preprint on the arXiv in December 2015) doubt has been cast on this intuition by the following two very similar theorems.

\begin{theorem}
\label{heuristicapplication}
For all $\varepsilon >0$ and all $k\geq 3$, there exists $N \in \N$ such that
for all coprime integers $m$ and $n$ with $n > N$ and $\ln \ln n < m < \ln n$,
there exists $f \in \Sym( \Z /n\Z)$ of order dividing $k$ with
$
f(x+1) = m f(x)
$
for at least $(1-\varepsilon)n$ values of $x \in \Z/n\Z$.
\end{theorem}

\begin{theorem}[Helfgott--Juschenko~\cite{he-ju}, also Glebsky~\cite{gl}]
\label{GS version}
For all $m>2$ and all $\ep >0$, there exists $C$ such that for all $n>C$ coprime to $m$, there exists $f \in \Sym(\Z /n\Z)$
of order dividing $4$ with $f(x+1) = m f(x)$ for at least $(1-\ep)n$ values of $x \in \Z /n\Z$.
\end{theorem}

Theorems~\ref{heuristicapplication} and~\ref{GS version} both run counter to Helfgott and Juschenko's heuristics.  (However, neither theorem addresses the case $m=2$ directly, so the existence of the functions $f$ of Helfgott and Juschenko's theorem remains open.)

We will prove Theorem~\ref{heuristicapplication} in Section~\ref{ZwrZ section}. It will be apparent there that we could replace $\ln \ln n$ and $\ln n$ with other functions.

Theorems~\ref{HJ thm}--\ref{GS version} all arise from a relationship between soficity and the existence of particular permutations of $\Z/n\Z$ set out in Theorem~\ref{main-vague} below, which is a generalization of a result of Helfgott and Juschenko \cite{he-ju}.  In the case of Theorem~\ref{HJ thm}, soficity of $H_4$ is a hypothesis.  For Theorem~\ref{heuristicapplication}, we use the soficity of $\H_4(\Z \wr \Z)$ established as a consequence of Theorem~\ref{ressolvable}.  Theorem~\ref{GS version} uses a theorem of Glebsky~\cite{gl} which says that for $m \geq 3$, $\overline{\H_4}(\BS(1,m))$ has sofic quotients into which $\BS(1,m)$ embeds.

\begin{theorem} \label{main-vague}
Suppose $G$ is a group, $\phi$ is an isomorphism $B \to A$ between subgroups $A,B \leq G$, and $k \geq 1$ is an integer. The following two conditions are equivalent.
\begin{enumerate}

\item \label{main-vague2}
$\H_k(G,\phi)$ has a sofic quotient $Q$ such that the composition $G \to \H_k(G,\phi) \to Q$ is injective.

\item \label{main-vague1}
Sofic approximations of $G$ exist for which there are permutations of order dividing $k$ that almost conjugate the action of $A$ to the action of $B$.

\end{enumerate}
If $G$ is amenable, then these are also equivalent to:

\begin{enumerate}
\addtocounter{enumi}{2}

\item \label{main-vague3}
For all sofic approximations of $G$ into sufficiently large symmetric groups, there are permutations of order dividing $k$ which almost conjugate the action of $A$ to the action of $B$.
\end{enumerate}
\end{theorem}

We will present a precise version of this theorem in Section~\ref{generalcase proof}.

The natural map $G \to \H_k(G, \phi)$ employed in \emph{(\ref{main-vague2})} can fail to be injective.  Indeed, it is rarely injective when $k$ is $1$ or $2$. The case $k=3$ is delicate.
As for when $k \geq 4$, in Lemma~\ref{injective} we will give sufficient conditions for injectivity  and in Example~\ref{non-injective} will show that injectivity can fail.

We will prove Theorem~\ref{main-vague} in Section~\ref{generalcase proof}, building on the arguments in~\cite{he-ju}.  The equivalence between Conditions~\emph{(\ref{main-vague2})} and~\emph{(\ref{main-vague1})} is analogous to that between the  two definitions of soficity   outlined at the start of this article---see Proposition~\ref{equiv}.
The idea behind the implication~\emph{(\ref{main-vague1})}~$\thus$~\emph{(\ref{main-vague2})} is that the sofic approximations together with the almost-conjugating functions  can be assembled into a homomorphism from $\H_k(G,\phi)$  to an ultraproduct of finite symmetric groups with image $Q$. For the implication~\emph{(\ref{main-vague2})}~$\thus$~\emph{(\ref{main-vague1})}, we obtain the requisite sofic approximation  of    $S \subseteq G$ and the almost-conjugating  permutation from a sofic approximation for   the  image  in  $Q$ of a   suitably constructed finite subset $S' \subseteq \H_k(G,\phi)$ with    $S   \cup \set{t} \subseteq S'$.

The  equivalence of~\emph{(\ref{main-vague3})} is significantly more complicated.
The additional assumption that the group $G$ is amenable gives better control of the sofic approximations.  The key result   is a theorem which is  due to Helfgott and Juschenko~\cite{he-ju} in the   form we will use and has origins in Elek and Szabo~\cite{es2} and Kerr and Li~\cite{kl}.
It spells out a manner in which any two sofic approximations of an amenable group are almost conjugate.

In  Section~\ref{examples} we  give applications of Theorem~\ref{main-vague}.  We look  at $G=\Z^2= \langle a,b \mid ab=ba \rangle$ and $\phi: b \mapsto a$, which we view as an introductory example---in this case, $\H_k(G, \phi)$ will be a  right-angled Artin group.  We     review the case of $G=\BS(1,m) = \langle a,b \mid a^b =a^m \rangle$ addressed by Helfgott and Juschenko and by Glebsky,  where soficity of $\H_k(G, \phi)$ remains unknown for $m \geq 2$.  We present our most novel applications which are when $G$ is the  3-dimensional integral Heisenberg group $\Heis$,  or $\Z \wr \Z$, or  the free metabelian group $\mathcal{M}$ on two generators.  In these cases, $\H_k(G, \phi)$ will be  sofic by Theorem~\ref{ressolvable}. We explain how the $G= \Z \wr \Z$ case leads to Theorem~\ref{heuristicapplication}.

We do not know how to construct functions $f$ explicitly  satisfying the  conditions of Theorems~\ref{heuristicapplication} or~\ref{GS version}. In principle one could follow the constructions  in the proofs, however this would require  constructing several F{\o}lner sets for  $G$ and switching between sofic approximations several times. (In the case of  Theorem~\ref{heuristicapplication}, where the quotients could be taken to be the metabelian groups of Proposition~\ref{soficquotient}, sofic approximations could be constructed explicitly; for Theorem~\ref{GS version} the quotients are residually nilpotent and constructing explicit  sofic approximations is again possible, but significantly more difficult.)
It seems unlikely that this will lead to an enlightening description of $f$.

By the same token, we do not know how $C$ and $N$ depend  on $\varepsilon$ in Theorems~\ref{HJ thm}--\ref{GS version}. One could obtain explicit estimates from our proofs, but they   will be very weak.  We   give some examples in Remarks~\ref{dependance1}, \ref{dependance2}, and \ref{dependance3}.  Sufficiently strong  estimates (which may well not exist) could have important applications, including a proof that Higman's group $H_4$ is sofic.

\section{Soficity}
\label{soficity section}

The \emph{normalized Hamming distance} $d$  on the symmetric group $\Sym(n)$   is
$$
d(\rho,\sigma) \  = \   \dfrac 1n |\{ 1 \le i \le n \mid \rho(i) \ne \sigma(i) \}|.
$$
This metric is invariant under both the left and right  action of $\Sym(n)$---i.e.,
$$
d(\rho,\sigma) \ = \  d(\tau \rho \tau' ,\tau \sigma \tau' )
$$
for all $\rho,\sigma, \tau, \tau' \in \Sym(n)$.  It follows that:

\begin{lemma} \label{hamming lemma}
\renewcommand{\labelenumi}{\textup{(\roman{enumi})}}
For $\sigma, \tau, \mu, \sigma_1, \ldots, \sigma_m \in \Sym(n)$,
\begin{enumerate}
\item $d( \id,   \sigma_1 \cdots  \sigma_m) \leq \sum_{i=1}^m   d( \id,   \sigma_i)$, \label{hammingi}
\item $d(\tau^{-1} \sigma \tau, \id) = d(\sigma, \id)$,  \label{hammingii}
\item $d(\tau^{-1} \sigma \tau, \mu^{-1} \sigma \mu) \leq 2 d(\tau, \mu)$. \label{hammingiii}
\end{enumerate}
\end{lemma}

For $n \in \N$, $\d >0$, and  $S$  a finite subset of a group $G$, an \emph{$(S,\d,n)$-approximation}  is a map $\psi:  G \to \Sym(n)$ such that
\begin{itemize}
\item $d(\psi(g)\psi(h),\psi(gh)) < \delta$ for all $g,h \in S$ such that $gh \in S$, and
\item $d(\psi(g),\mathrm{id}) > 1 - \delta$ for all $g \in S \ssm \set{e}$.
\end{itemize}

(That $\psi$ is defined on all of $G$,  instead of just on $S$, is a technical convenience.  Its values on  $G \ssm S$ are irrelevant  to the definition.)

A \emph{filter} $\mathcal F$ on a set $I$ is a nonempty set of subsets of $I$ such that $\emptyset\not\in \mathcal F$; for all $U, V \in \mathcal U$, $U \cap V \in \mathcal F$; and if $U \in \mathcal F$ and $U \ins V$, then $V \in \mathcal F$. An \emph{ultrafilter} $\mathcal U$ on $I$ is a maximal filter; equivalently, for all $U \ins I$, either $U \in \mathcal U$ or $(I \smallsetminus U) \in \mathcal U$.

Suppose $\mathcal U$ is an ultrafilter on a set $I$.  To each ${i} \in I$ associate some $n_{i} \in \N$.  For    $x = (x_{i})_{i \in I}$ and $y = (y_{i})_{i \in I}$ in the direct product
$\prod_{i \in I} \Sym(n_{i})$, we write $x \approx_{\mathcal U} y$ when $\{i \in I \mid d(x_{i},y_{i}) < \delta\} \in \mathcal U$ for all $\delta > 0$.
Let $\id = (\id_{n_{i}})_{i \in I}$. Define $\mathcal N := \set{ x \in \prod_{i \in I} \Sym(n_{i}) \mid x \approx_{\mathcal U} \id  } $, which is called the normal subgroup of \emph{infinitesimals}. Define  the \emph{(metric) ultraproduct} $\prod_{\mathcal U} \Sym(n_{i}) := \left(\prod_{i \in I} \Sym(n_{i}) \right) \big/ \mathcal N$.    See \cite{pe} for further background.

A group $G$ is \emph{sofic} when it satisfies either of the conditions of the following proposition.

\begin{proposition}
\label{equiv}
For a group $G$, the following are equivalent.
\begin{enumerate}
\item \label{sof3}
    The group $G$  is isomorphic to a subgroup of some metric ultraproduct of finite symmetric groups---that is, there exist  an ultrafilter $\mathcal U$ on a set $I$, natural numbers $\set{n_{i}}_{i \in I}$, and an injective homomorphism
$$
G \into \prod_{\mathcal U} \Sym(n_{i}).
$$
\item \label{sof1}
    For all finite subsets $S \ins G$ and   all $\d > 0$, there exists an $(S,\d,n)$-approximation for some $n$.
\end{enumerate}
\end{proposition}

\begin{proof}
Here is a sketch.  Details are in~\cite{pe,es}.

For~\emph{(\ref{sof3})}~$\thus$~\emph{(\ref{sof1})},
 a  homomorphic embedding $G \into \prod_{\mathcal U} \Sym(n_{i})$  can be lifted (non-uniquely) to a map $\psi = (\psi_i): G \to \prod_{i \in I} \Sym(n_{i})$, where $\psi_i: G \to \Sym(n_{i})$. However, $\psi$ may fail to be a group homomorphism. For  all  $a,b \in G$,
   $\psi(a) \psi(b) \psi(ab)^{-1}$ is an infinitesimal.  This implies that for each finite set $S$ and each $\d > 0$, the set of $i$
such that $\psi_i$  is an $(S,\d,n_i)$-approximation is in the ultrafilter $\mathcal{U}$, and so is not empty.
The second condition of the   approximation is not immediately satisfied---one only gets that $d(\psi_i(g), \id) > \d$ for $g \in S \ssm \set{e}$. An `amplification trick'
improves this to  $1 -\d$.

For \emph{(\ref{sof1})}~$\thus$~\emph{(\ref{sof3})},  let $I = \set{(S, \delta) \mid \text{ finite } S \subseteq G, \ \delta >0 }$.  For $(S, \delta) \in I$, define $$\overline{(S, \delta)} \ := \  \set{(S', \delta') \in I \mid \text{ finite } S' \supseteq S, \ \delta' \leq \delta }.$$
The family  $\mathcal{F}$ of all subsets $\overline{(S, \delta)}$  of $I$ where $(S, \delta) \in I$  enjoys the finite intersection property since $\bigcap_{i=1}^k  \overline{(S_i, \delta_i)}  \ = \  \overline{(\bigcup_{i=1}^k S_i, \max_{i=1}^k \delta_i)}$. So there is  an ultrafilter $\mathcal{U}$ on $I$ with $\mathcal{F} \subseteq \mathcal{U}$.  For all $i = (S, \delta) \in I$,  let $\psi_i: G \to \Sym(n_i)$ be an $(S,\d,n_i)$-approximation.   These maps combine in   $g \mapsto (\psi_i(g))_{i \in I}$ to induce a monomorphism $G \into \prod_{\mathcal U} \Sym(n_{i})$: it is a homomorphism because  for all $g, h \in G$, $$(\psi_i(g)\psi_i(h)\psi_i(gh)^{-1})_{i \in I} \in \mathcal{N}$$ since for all $\delta >0$, $\set{i \in I \mid d(\psi_i(g)\psi_i(h),\psi_i(gh)) <\delta} \in \mathcal{U}$ as it is a superset of $\overline{(\set{g,h,gh}, \delta)} \in \mathcal{U}$; and it is injective because likewise for $\delta >0$ and $g \in G \ssm \set{e}$, the set $$\set{i \in I \mid d(\psi_i(g), \id_{n_i}) > 1-\delta } \in \mathcal{U}$$  and so  $(\psi_i(g))_{i \in I} \notin \mathcal{N}$.
\end{proof}

Given  that the   formulation~\emph{(\ref{sof1})}  of soficity is in terms of finite subsets of $G$, it is immediate  that a group is sofic if and only if its finitely generated subgroups are sofic.  Also, subgroups of finite symmetric groups are sofic, so all finite groups are sofic. This generalizes as follows. A group $G$ is \emph{residually P} if for every   $x \in G \ssm \set{e}$, there is some  quotient $\ph_x : G \onto H_x$ such that $\ph_x(x)$ is not trivial and $H_x$ satisfies condition $P$.  Residually finite groups are sofic:   if $S  \subseteq G$ is finite and $\ph_x : G \onto H_x$  are as per the definition with $H_x$ finite, then   $\ph :=  \bigoplus_{x \in S \ssm \set{e}} \ph_x$ is a faithful map to a finite group and composes with a map to  some $\Sym(n)$ to give an $(S,0,n)$-sofic approximation. More  generally, residually sofic groups are sofic.

Amenable  groups are also sofic.
A group $G$ is \emph{amenable} when it satisfies the \emph{F\o lner condition}:
for all finite subsets $S \ins G$ and for all $\ep > 0$, there is a finite subset $\Phi \ins G$ such that for each $g \in S$, $|g\Phi \triangle \Phi| < \ep|\Phi|$ (where $\triangle$ denotes symmetric difference: $A \triangle B = (A \ssm B) \cup (B \ssm A) = (A \cup B) \ssm (A \cap B)$).   For every $g \in S$, the map $\Phi \to \Phi$ given by $x \mapsto gx$ is well-defined on all but $\ep|\Phi|$ elements of $\Phi$. Extend to the rest of $\Phi$ arbitrarily so that the map is a bijection, and then each element of $g$ corresponds to an element of the symmetric group $\Sym(|\Phi|)$. The function identifying each $g$ with the corresponding map gives an $(S,2\ep,|\Phi|)$-approximation of $G$.  More details are in \cite{pe}.

It then follows that residually amenable and, in particular, residually solvable groups are sofic (a fact we will use for Corollary~\ref{so sofic cor}).

The class of sofic groups  enjoys various closure properties.  These are all sofic: graph products (e.g.\ free or direct products) of sofic groups  \cite{chr}, amalgamated free products or HNN extensions of sofic groups  over amenable groups  \cite{co-dy,  es2, pu},   wreath products of sofic groups \cite{hs}, groups with finite index sofic subgroups,   limits of sofic groups in the space of marked groups (but not all finitely generated sofic groups are limits of amenable groups \cite{co}),  locally sofic groups (e.g.\ groups  locally embeddable into sofic groups or  direct limits of sofic groups).    If $G$ has a sofic normal subgroup $N$ such that $G/N$ is amenable, then $G$ is sofic.  Whether the same conclusion can be drawn when $N$ is amenable and $G/N$ is sofic, is open.

Soficity relates to a number of outstanding open problems.   In 1973 Gottschalk defined a group $G$ to be \emph{surjunctive} when for every finite set $S$ and for $S^G$ the set of functions $G \to S$, every continuous $G$-equivariant injective function $f: S^G \to S^G$ is also surjective.  Gottschalk  conjectured  that all groups are surjunctive.  A group is hyperlinear when every finite subset can be well approximated in a unitary group with the normalized Hilbert--Schmidt norm.     Connes' Embedding Conjecture states that every group is hyperlinear.
Kaplansky's Direct Finiteness Conjecture is that if  $G$ is a group and  $K$ is a field and  if $a, b \in K[G]$ satisfy $ab=1$, then $ba=1$.
Sofic groups are surjunctive~\cite{gr, we},  are hyperlinear (since finite permutation groups embed in unitary groups), and satisfy the Direct Finiteness Conjecture~\cite{es}.   The most recent progress is the construction by De~Chiffre, Glebsky,  Lubotzky, and Thom  of  groups that do not satisfy an alternate version of the hyperlinear condition where the Hilbert--Schmidt norm is not  normalized  \cite{cglt}.

For further background, we refer to the surveys~\cite{ca-lu,pe}.

\section{Variations on Higman's group}
\label{Generalizing section}

Our notation is $b^a = a^{-1} b a$ and $[a,b] = a^{-1} b^{-1} a b$.

As explained in Section~\ref{results}, for a group $G$, subgroups $A$ and $B$, an isomorphism  $\phi: B \to A$, and a $k \in \N$, we define $G_i$, where $1 \leq i \leq k$, to be copies of $G$ and  $A_i, B_i \leq G_i$ to be copies of its subgroups $A$ and $B$.  Then $\phi$ induces an isomorphism  $\phi_i: B_i \to A_{i+1}$   and we define
$$
\overline{\H}_k(G, \phi)  \  := \ \langle G_1, \ldots, G_k \mid  b_{i} = \phi_i(b_i) \text{ for all } i \text{ and all } b_i \in B_i  \rangle.
$$
Thus, $\overline{\H}_k(G, \phi)$ is the quotient of the free product of $k$ copies of $G$ in which $B$ in the $i$-th is identified with $A$ in the $(i+1)$-st   for $i = 0, \ldots, k-1$ (indices modulo $k$).

By construction there are maps $\iota_1,\dots, \iota_k$ from the group $G$ to $\overline \H_k(G,\phi)$.   We regard $\iota:=\iota_1$ as the \emph{natural map} $G \to \overline \H_k(G,\phi)$.  We will often work in settings where these maps are injective, and then for simplicity we will suppress them and consider $G$ as a subgroup of $\overline \H_k(G,\phi)$ via  $\iota$.

For example, if $G = \langle a_1, a_2 \mid R \rangle$ is a 2-generator group such that  $a_1$ and $a_2$ have the same order,  then  $\overline{\H}_k(G, \phi)$, where
$\phi : a_2 \mapsto a_1$, is the \emph{cyclically presented group}
$$
\langle a_1, \ldots, a_k \mid \sigma^i(r) \, ; \  r \in R, \  i = 0, \ldots, k-1 \rangle,
$$
where $\sigma$ cycles    the indices of the letters of $r$.

The semi-direct product of  $\overline{\H}_k(G, \phi)$ with the cyclic group $C_k$ of order $k$ in which a generator $t$ of $C_k$ conjugates $G_i$ to $G_{i+1}$ (indices mod $k$) is
$$
\H_k(G, \phi)   \ = \ \langle G,   t \mid  t^k = 1,  \ b^t = \phi(b);  \forall b \in B  \rangle.
$$
Then $\overline\H_k(G, \phi)$ is the normal closure of  $\iota(G)$  in $\H_k(G, \phi)$   and is the kernel of $\H_k(G, \phi) \to C_k$.

In the case when $G$ is a group generated by two elements $a,b \in G$ of the same order, with $A = \langle a \rangle$, $B = \langle b \rangle$, and $\phi:B \to A$ given by $\phi(b) = a$, we will  write $\H_k(G)$  and  $\overline \H_k(G)$ in place of $\H_k(G,\phi)$   and $\overline \H_k(G,\phi)$.

The cases $k=1,2$ are degenerate:
\begin{lemma}
$\overline{\H}_1(G, \phi)$ is a quotient of $G$.  If $G$ is generated by the subgroups $A$ and $B$, then $\overline{\H}_2(G, \phi)$ is a quotient of $G$.
\end{lemma}

For large $k$ one expects   $G$ generally to embed in $\H_k(G, \phi)$, but this can fail:

\begin{ex}
\label{non-injective}
When $G=B=\Z$, $A=2\Z$, and $\phi$ is multiplication by $2$,  $\H_k(G,\phi)$ is finite for all $k$, and so
$\iota: G \not\hookrightarrow
  \H_k(G, \phi)$.
\end{ex}

When $k\geq 4$,  here is a sufficient condition:
\begin{lemma}
\label{injective}
If $A \cap B = \set{1}$ and $k \geq 4$, then $G$ and $A \ast A$  both embed in $\H_k(G,\phi)$. In particular, if $G \neq \set{1}$, then  $\H_k(G,\phi)$ is not amenable.
\end{lemma}
\begin{proof}
Let $J = G_1 \ast_{\phi_1} G_2 \ast_{\phi_2} \cdots \ast_{\phi_{k-3}} G_{k-2}$, and let $K = G_{k-1} \ast_{\phi_{k-1}} G_k$. Since $A_1 \cap B_1 = A_1 \cap A_2 = \{1\}$  and   $A_2 \cap B_2 = B_1 \cap B_2 = \{1\}$, the subgroup generated by $A_1$ and $B_2$ in $G_1 \ast_{\phi_1} G_2$ is   $A_1 \ast B_2$. Inductively, the same holds for the subgroup generated by $A_1$ and $B_{k-2}$ in $J$, and similarly for the subgroup generated by $A_{k-1}$ and $B_k$ in $K$. Then $\H_k(G,\phi)$ is the amalgamated free product of $J$ and $K$ along the subgroup $\langle A_1,B_{k-2} \rangle = A_1  \ast   B_{k-2}$, which is identified with $B_k \ast A_{k-1}$ via identifying $A_1$ with $B_k$ and   $B_{k-2}$ with $A_{k-1}$. Thus $A \ast A$ embeds in $\H_k(G,\phi)$ since $A \ast A \cong A_1 \ast B_{k-2} \le G$. Meanwhile $G_1 \le J$, so $G_1 \le \H_k(G,\phi)$ as well,  and  the canonical map $G \to \H_k(G, \phi)$ is   injective.

If $A \neq \set{1}$ then the subgroup  $A \ast A$ prevents $\H_k(G, \phi)$  from being amenable.  If $A = \set{1}$, then $\H_k(G)$ is  a free product.
\end{proof}

The case $k=3$ is trickier.  Sometimes $G$ does not embed in $\H_3(G,\phi)$ because the latter group is very small---for example,  $\overline{\H}_3(\textup{BS}(1,2)) = \set{1}$---but it is also possible that $G$ embeds in $\H_3(G,\phi)$, which is the case for most other examples considered in this paper.

\section{Soficity via residual solvability}
\label{resfinite section}

Here we prove Theorem~\ref{ressolvable} by an approach which is similar to our proof  of  Lemma~\ref{injective}: it is based on viewing the amalgamated products as a combination of a free product and a semidirect product.

We have  that $G$ is residually solvable and has subgroups $A$ and $B$ for which there is an isomorphism $\phi: B \to A$, and that there exists a group homomorphism
$\pi: G \to A \times B$ such that  $\pi(a) = (a,1)$ for all $a \in A$ and $\pi(b) = (1,b)$ for all $b \in B$.  We aim  to
show that $\H_k(G, \phi)$  and $\overline{\H}_k(G, \phi)$ are also residually solvable for all $k \geq 4$.

Define  $G_A= \pi^{-1}(1,*)$ or, equivalently,
$G_A = \ker ( \phi_A \circ \pi)$, where $\phi_A$ is the projection   $A\times B \to A$. So $G_A$ is a normal subgroup of $G$ and $G/G_A \simeq A$.
The hypothesis that  $\pi(a) = (a,1)$ for all $a \in A$   implies that $A$ is a complement of $G_A$ in $G$, and so $G$ can be expressed as a  semidirect product $G = A \ltimes G_A$.
And  $B \subseteq G_A$  because $\pi(b) = (1,b)$ for all $b \in B$.
Likewise, $G = B \ltimes G_B$ with    $A\subseteq G_B$.

 As (3)--(5) of the following examples show, the hypotheses of  Theorem~\ref{ressolvable} do not imply that $A$ and $B$ commute.  Rather, they imply that $[A,B] \subseteq G_A \cap G_B = \ker \pi$.

\begin{Examples}
\label{example groups}
In each case take $A= \langle a \rangle = \Z$ and $B= \langle b \rangle = \Z$:
\begin{enumerate}
\item  $G=\Z^2 = \langle a, b \mid ab=ba \rangle$.  Take $\pi$ to be the identity.  The semi-direct products are   direct products $\Z \times \Z$.
\item  $G=\BS(1,m) = \langle a, b \mid a^b=a^m \rangle$.  In this case there is no map $\pi$ for $m \neq 1$ because $[a,b] =a^{m-1}$, and so cannot be in $\ker \pi$.
\item $G = \Heis = \langle a, b \mid [a,[a,b]] = [b,[a,b]] =1 \rangle$, the three-dimensional integral Heisenberg group.  Take $\pi$ to be the map onto $\Z^2 = \langle a, b \mid ab=ba \rangle$ quotienting by the center $\langle [a,b] \rangle$ of $\Heis$.  Then $G_A = \langle b, [a,b] \rangle \simeq  \Z^2$ and $G_B = \langle a, [a,b] \rangle \simeq  \Z^2$.
\item $G = \Z \wr \Z = \left\langle a, b \, \left|  \,  \left[a^{b^i}, a^{b^j}\right]=1 \text{ for all  }  i, j   \right.  \right\rangle$, which is $\Z   \ltimes \bigoplus_{i \in \Z} \Z =    \langle b \rangle \ltimes \bigoplus_{i \in \Z} \langle a_i  \rangle$ where $a_i = a^{b^i}$ and $b$ acts so as to map $a_i \mapsto a_{i+1}$.  Again, take $\pi$ to be the abelianization map onto $\Z^2 = \langle a, b \mid ab=ba \rangle$.   Then $G_B$ is the kernel of the map  $\Z \wr \Z \onto \langle b \rangle$ given by quotienting by $a$,  which is $\bigoplus_{i \in \Z} \Z =     \bigoplus_{i \in \Z} \langle a_i  \rangle$.
And $G_A$  is the kernel of the map  $\Z \wr \Z \onto \langle a \rangle$ given by quotienting by $b$, which is   $\langle b \rangle \ltimes \bigoplus_{i \in \Z} \langle a_i^{-1}a_{i+1}  \rangle$ and is isomorphic to $G$.
\item  $G =  \Met = \left\langle a, b \  \left| \  \left[[a, b],[a,b]^{a^ib^j}\right] = 1 \ \forall i,j \in \Z \right. \right\rangle
$, the free metabelian group on two generators.  Again, we take $\pi$ to be the abelianization map onto $\Z^2 = \langle a, b \mid ab=ba \rangle$.
\end{enumerate}
\end{Examples}

We will use the following  description of amalgamated products over subgroups which have a normal complement.
\begin{lemma}
\label{lm-amalgamated-as-semidirect}
Suppose $G_1$ and $G_2$ are groups having subgroups $H_1$ and $H_2$ respectively with normal complements---i.e., $G_1 = H_1 \ltimes N_1$ and
$G_2 = H_2 \ltimes N_2$ for some $N_1$ and $N_2$. For any isomorphism $\phi: H_1 \to H_2$,
the amalgamated product $G_1 \ast_\phi G_2$ can be expressed as a semidirect product  $H_1 \ltimes (N_1 \ast N_2)$, where the action of $H_1$ on $N_2$ comes from that of $H_2$ via the isomorphism $\phi$.
\end{lemma}
\begin{proof}
An arbitrary element of $G_1 \ast_\phi G_2$ can be represented as a product
$$
w = x_1 y_1 x_2 y_2 \dots x_r y_r,
$$
where $x_1, \ldots, x_r \in G_1$ and $y_1, \ldots, y_r \in G_2$.
Express $x_1$ as $m_1 h_1$ where $m_1 \in G_1$ and $h_1 \in h_1$. Since we are working in the amalgamated product, we can move $h_1$ to  $G_2$ and  write
$$
w = m_1 \left(\phi(h_1) y_1\right) x_2 y_2 \dots x_r y_r.
$$
The element $\phi(h_1) y_1$ in $G_2$ can then be expressed as $n_1 g_1$ where $n_1 \in N_2$ and $g_1 \in H_2$.
Continuing this process,  moving elements from $H_1$ or $H_2$ to the right,  expresses $w$ as
\begin{equation} \label{prod expression}
w = m_1 n_1  m_2 n_2  \dots m_r n_r h,
\end{equation}
where $m_1, \ldots, m_r \in N_1 \subseteq G_1$,  $n_1, \ldots, n_r \in N_2 \subseteq G_2$, and $h \in H_1$.
The product $m_1 n_1  \dots m_r n_r$ can be considered as an element in $N_1 \ast N_2$.  Such elements form a normal subgroup in $G_1 \ast_\phi G_2$, with quotient $H_1$. All that remains to check is that the action of $H_1$ on the free product is the one described.
\end{proof}

\begin{corollary}
\label{lm-amalgamated-is-res-solvable}
Suppose $G_1$ and $G_2$ are residually solvable groups   satisfying the conditions of Lemma~\ref{lm-amalgamated-as-semidirect}.
Then the amalgamated product $G_1 \ast_\phi G_2$ is residually solvable.
\end{corollary}
\begin{proof}
Free products of residually solvable groups are residually solvable, but semidirect products of residually solvable groups can fail to be  residually solvable.  Nevertheless we will see that   the semidirect products of Lemma~\ref{lm-amalgamated-as-semidirect} are residually solvable.

Let $w = m_1 n_1  m_2 n_2  \dots m_r n_r h$ be a non-trivial element in $G_1 \ast_\phi G_2$ as per~\eqref{prod expression},
where all $m_i \in G_1$ and $n_i\in G_2$ are non-identity, with the possible exceptions of $m_1$  and $n_r$, with $h \in H_1$. If $h \not = 1$, then there is a solvable quotient $H_1 \to \bar H$  of $H_1$ where
$h$ survives (since subgroups of residually solvable groups are residually solvable), which leads to a quotient $G_1 \ast_\phi G_2$, where $w$ has a nontrivial image. Therefore it suffices to consider the case $h=1$.

Take $k$ such that for $i=1,2$,  $G_i \to \bar G_i := G_i/ G_i^{(k)}$ are quotients of $G_i$ by some derived subgroup such that all the $m_i$ and $n_i$ have nontrivial images in $\bar G_1$ and $\bar G_2$.
Let $\bar{H}_1$, $\bar N_1$ and $\bar N_2$ denote the (necessarily solvable) images  of $H_1$, $N_1$ and $N_2$, respectively, in $\bar G_1$ and $\bar G_2$.
We can view $w$ as element in the free product $\bar G_A \ast \bar G_B$. Therefore,  by the argument that free products of solvable groups are residually solvable (see e.g.~\cite{g}), there exists a quotient
$\overline{\bar N_1 \ast \bar N_2}$ of $\bar N_1 \ast \bar N_2$ by one of its derived subgroups where $w$ is non-trivial. Since this quotient is characteristic, it has a natural action of $\bar H_1$ which extends the actions of $\bar H_1$ on $\bar N_1$ and on $\bar N_2$.  This allows us to  map
$$
G_1 \ast_\phi G_2 = H_1 \ltimes (N_1 \ast N_2) \to \bar H_1 \ltimes ( \bar N_1 \ast \bar N_2 ) \to \bar H_1 \ltimes \overline{\bar N_1 \ast \bar N_2},
$$
where $\bar H_1 \ltimes \overline{\bar N_1 \ast \bar N_2}$ is  a solvable quotient of $G \ast_\phi G$ in which $w$ has a nontrivial image.
\end{proof}
\begin{lemma}
\label{lm-amalgamated-mapsto}
Suppose $G$ is a  group satisfying the conditions in Theorem~\ref{ressolvable}.
Then the amalgamated product $G \ast_\phi G$ can be written as a semidirect product $(A \ast B) \ltimes H$ for some normal subgroup $H$, and therefore
there is a projection   $G \ast_\phi G \to A \ast  B$.
\end{lemma}
\begin{proof}
Annihilating the first factor in $B \ltimes (G_A \ast G_B)$ and then using the maps $G_A \to B$ and $G_B \to A$ induced by $\pi$, maps $G \ast_\phi G$ to  $A \ast B$.
This map is clearly surjective with some kernel $H$ and restricts to the identity on $A \ast B$ (viewed as a  subgroup of $G_A \ast G_B$  via   $B \leq G_A$ and $A \leq G_B$), so splits $G \ast_\phi G$ into a semidirect product.
\end{proof}

\begin{proof}[Proof of Theorem~\ref{ressolvable}]
Applying   Corollary~\ref{lm-amalgamated-is-res-solvable} and Lemma~\ref{lm-amalgamated-mapsto} repeatedly, we find that if $k \geq 4$, then the groups
$J := G_1 \ast_{\phi_1} G_2 \ast_{\phi_2} \cdots \ast_{\phi_{k-3}} G_{k-2}$  and   $K := G_{k-1} \ast_{\phi_{k-1}} G_k$ (in the notation of Section~\ref{Generalizing section}) are both residually solvable,  and both contain $A \ast B$ in such a way that they both
split over this group as semidirect products, and   $\overline{\H}_k(G,\phi)  =  J \ast_{A \ast B} K$.
So the hypotheses of Lemma~\ref{lm-amalgamated-as-semidirect} are met and $\overline{\H}_k(G,\phi)$  is residually finite by a final application of Corollary~\ref{lm-amalgamated-is-res-solvable}.

Finally, $\H_k(G, \phi) = \overline{\H}_k(G, \phi) \rtimes C_k$, so is also residually solvable. (Semidirect products $H \rtimes A$  of  residually solvable groups $H$ and solvable groups $A$ are residually solvable.)
\end{proof}

Theorem~\ref{ressolvable} may also hold  when `residually solvable' is replaced with `residually nilpotent' or `residually finite'; however, our proof would need further ideas and the given theorem suffices for our application:

\begin{corollary} \label{so sofic cor}
When  $G$ is $\Z^2$, $\mathcal{H}$,  $\Z \wr \Z$, or $\Met$ as per Examples~\ref{example groups},  $\overline{\H}_k(G)$ is residually solvable, and so sofic, for all $k \geq 4$.
\end{corollary}

Finally, we remark on an alternative route:

\begin{proposition}
\label{soficquotient}
Suppose there exists a homomorphism $\pi : G \to A \times B$ as per Theorem~\ref{ressolvable}.
Then for all $k \geq 3$, there are  homomorphisms $\mu: \H_k(G, \phi) \to C_k \ltimes G^k$ and $\bar \mu: \bar \H_k(G, \phi) \to G^k$.
Moreover, the restrictions  of $\mu$ and $\bar{\mu}$ to any copy $G_i$ of $G$ inside $\H_k(G)$ are injective.
\end{proposition}
\begin{proof}
Let $\pi_A : G \to A$ (respectively, $\pi_B : G \to B$) be the composition  of $\pi$ with projection onto $A$ (respectively, $B$).
Define the homomorphism $\bar \mu: \bar \H_k(G, \phi) \to G^k$, given by
$$
\overline{\mu}(\iota_l(g)) = (1,\dots,1, \phi^{-1}(\pi_A(g)), g, \phi(\pi_B(g)),1,\dots,1),
$$
where $g$ is an arbitrary element of $G$, and $\iota_l(g)$ is the element in $\overline \H_k(G)$ corresponding to $g$ sitting in the $l$-th copy of $G$. The elements
$\phi(\pi_B(g))$, $g$, and $\phi^{-1}(\pi_A(g))$ are sitting in coordinates $l-1$, $l$ and $l+1$.
Clearly $\overline{\mu}$ is well defined on each copy $G_i$ appearing in the presentation of $\overline\H_k(G)$, so it suffices to verify that
$\overline{\mu}$ identifies the $l$-th copy of $B$ with the $l+1$-st copy of $A$. By definition we have
$$
\overline{\mu}(\iota_l(b)) = (1,\dots,1, \phi^{-1}(\pi_A(b)), b, \phi(\pi_B(b)),1,\dots,1) =
(1,\dots,1, 1, b, \phi(b),1,\dots,1)
$$
$$
\overline{\mu}(\iota_{l+1}(a)) = (1,\dots,1, \phi^{-1}(\pi_A(a)), a, \phi(\pi_B(a)),1,\dots,1) =
(1,\dots,1, \phi^{-1}(a) , a, 1 ,1,\dots,1),
$$
and thus
$ \overline{\mu}(\iota_l(b)) = \overline{\mu}(\iota_{l+1}(\phi (b)))$, i.e., $\overline{\mu}$ extends to the group $\H_k(G)$.
By construction, the restriction of $\overline{\mu}$ on each copy of $G$ is injective. (Unless we are in a degenerate case, the maps $\mu$ and $\overline{\mu}$
are not surjective.)\end{proof}

This is weaker than Theorem~\ref{ressolvable}
  in that it  does not tell us that $\H_k(G, \phi)$  is sofic or   residually solvable. But this proposition would suffice for our applications in Section~\ref{examples} because it tells is that when $G$ is sofic,  $\H_k(G, \phi)$ has a sofic quotient into which $G$ injects  (condition~(1) of Theorem~\ref{main-vague}).  Moreover, it  does so in a manner that makes   sofic approximations of that quotient easy to construct explicitly from sofic approximations of $G$.

\begin{remark}
If we remove the defining relator $t^k=1$ from our presentation for $\H_{k}(G,\phi)$, then $t$ becomes the stable letter of the HNN-extension   $\langle G,t \mid b^t  = \phi(b) \  \forall b \in B \rangle$, which is more straightforward to understand in the context of soficity.  For example, the instance where $G = \langle a ,b \mid  b^a = b^2 \rangle$ and $\phi: b \mapsto a$ is Baumslag's   one-relator group $\langle b,  t \mid  b^{b^t}  = b^2  \rangle$.
If $G$ is solvable, then $\langle G,t \mid b^t  = \phi(b) \  \forall b \in B \rangle$ is sofic:   Collins and Dykema~\cite[Corollary~3.6]{co-dy} show that an HNN-extension of a sofic group $G$ relative to an injective group homomorphism $\theta: H \to G$, for $H \le G$ \emph{monotileably amenable}, is sofic. If  $G$ is solvable, then so is its subgroup $B$.  Solvable groups are monotileably amenable, thereby implying  $\langle G,t \mid b^t  = \phi(b) \  \forall b \in B \rangle$ is sofic.
\end{remark}

\section{Sofic quotients and almost conjugation}
\label{generalcase proof}

This section is devoted to proving Theorem~\ref{main-vague} relating soficity in the context of $\H_k(G,\phi)$ to  seemingly pathological permutations $f$.  These $f$ come from permutations approximating   $t\in \H_k(G,\phi)$.  They will have order dividing $k$ since $t^k=1$ and, for all $b \in B$, will `almost conjugate' permutations approximating $b$ to permutations approximating $\phi(b)$ since $b^t = \phi(b)$ in $\H_k(G,\phi)$.  When $G$ is amenable and we have   explicit sofic approximations for $G$,  the permutations approximating $b$ and $\phi(b)$ in  $\H_k(G,\phi)$  essentially have to be those  sofic approximations.  In  examples, the `almost conjugate' conclusion  will then amount to   a local recurrence such as $f(x+1) = m f(x)$ holding for most values of $x$.

We make Theorem~\ref{main-vague} precise as:

\begin{theorem}
\label{main-precise}
Suppose $G$ is a group,   $\phi$ is an isomorphism $B \to A$ between subgroups $A,B \leq G$, and $k \geq 1$ is an integer.  The following two conditions are equivalent.
\begin{enumerate}
\item
\label{main-precise1}
$\H_k(G,\phi)$ has a sofic quotient $Q$ such that the  composition $G  \stackrel{\iota}{\to} \H_k(G,\phi) \to Q$ is injective.

\item
\label{main-precise2}
For all finite subsets $S \subseteq G$ and all $\d,\varepsilon >0$, there exists   an $(S,\d,n)$-approximation $\psi$  of $G$ and a permutation $f \in \Sym(n)$ of order dividing $k$ such that for all $b \in S  \cap \phi^{-1}(A \cap S)$,
$$
d( \psi(b) \circ f, f \circ \psi(\phi(b)))  \ < \    \varepsilon.
$$
\end{enumerate}
If  $G$ is amenable, then these are also equivalent to:

\begin{enumerate}
\addtocounter{enumi}{2}

\item  \label{main-precise3}
For all finite sets $S \subseteq G$ and all $\varepsilon >0$, there exist
a finite set $S' \subseteq G$ with $S \subseteq S'$ and  $\d >0$   and an integer $N$ such that if $\psi$ is an $(S',\d,n)$-sofic approximation of $G$ with $n> N$,
then there exists a permutation $f \in \Sym(n)$ of order dividing $k$
such that  for all $b \in S \cap \phi^{-1}(A \cap S)$,
$$
d( \psi(b) \circ f, f \circ \psi(\phi(b)))  \ < \    \varepsilon.
$$
\end{enumerate}

\end{theorem}

\begin{proof}[Proof of Theorem~\ref{main-precise},  (\ref{main-precise1})~$\thus$~(\ref{main-precise2})]
We have that there is a sofic quotient $Q$ such that the composition of the natural map $G \stackrel{\iota}{\to} \H_k(G,\phi)$ with the quotient map $\pi: \H_k(G,\phi) \to Q$ is injective.  In particular, the map $\iota$ is injective.

Suppose $S\subseteq G$ is a finite subset and  $\ep , \d > 0$. We seek an $n$ and an $(S, \d, n)$-approximation $\psi$ of $G$ together with a permutation $f \in \Sym(n)$ of order dividing $k$ such that $d( \psi(\phi(b)) \circ f, f \circ \psi(b)) <    \varepsilon$ for all $b \in S  \cap \phi^{-1}(A \cap S)$ .

Let
$$
S'  \ = \ \set{\id, t, \ldots, t^{k-1}} \cup  \iota(S) \cup    \iota\left(S  \cap \phi^{-1}(A \cap S) \right) t     \  \subseteq \  \H_k(G, \phi).
$$

  Let $\d'=\min\{\d,\varepsilon\}/ 6k$.
Then $\pi(S')$ is a finite subset of the sofic group $Q$,   so there exists  an $n \in \N$ and an
$(\pi(S'), \d' ,n)$-approximation $Q \to \Sym(m)$.  Via $\pi$ this gives a map $\psi': \H_k(G,\phi) \to \Sym(n)$ which enjoys the first defining property of an $(S', \d' ,n)$-approximation, but  may fail the second as it could map some elements of $S'$ to the identity.
Since $G$ naturally maps into $\H_k(G,\phi)$ and $\d' < \d$, the composition $\psi$ of $\iota$ and $\psi'$ is an $(S,\d,n)$-approximation of $G$, as required.

We will obtain the requisite  permutation $f \in \Sym(n)$ from  the action of $t$ under $\psi'$. First set
$\tilde f =  \psi'(t).$
The order of this permutation may fail to divide $k$ since $\psi'$ is not necessarily a homomorphism.  However,
$$
d(\tilde f^k , \id)  \ = \  d( \psi'(t)^k, \id) \ \leq \    d( \psi'(t)^k,  \psi'(t^k))  +  d( \psi'(t^k), \id)  \ < \    (k-1) \d'  +  \d'   \ = \  k \d',
$$
where the second inequality holds because $t^k =\id$ and $t^i \in S'$ for all $i$. Therefore  the set of points which are not part of a cycle of length dividing $k$ under the action of $\tilde f$ is correspondingly small and we can find a permutation $f$ of order dividing $k$ such that $d(f, \tilde f) < k \d' $.

Suppose $b \in S \cap \phi^{-1}(A \cap S)$.
It remains to show that
$$
d\left(\psi(b) \circ f, f \circ \psi(\phi(b))\right) \ \leq \  \varepsilon.
$$
As  $d(f, \tilde f) \leq k \delta'$,   Lemma~\ref{hamming lemma} (iii)  yields
\begin{equation} \label{first}
d\left(
f^{-1} \circ \psi(b) \circ f,
\tilde f^{-1} \circ \psi(b) \circ \tilde f
\right)  \ <  \  2k \d'.
\end{equation}
By definition of $\tilde{f}$,
\begin{equation}
\tilde f^{-1} \circ \psi(b) \circ \tilde f  \ = \   \psi'(t)^{-1} \circ  \psi'(\iota(b)) \circ  \psi'(t).
\end{equation}
Now,  as $t,  t^{-1}, \id \in S'$ and  $\psi'$ is an $(S', \d' ,n)$-approximation,
\begin{equation}
d\left( \psi'(t)^{-1}  \circ \psi'(\iota(b)) \circ   \psi'(t),  \psi'(t^{-1})  \circ \psi'(\iota(b)) \circ   \psi'(t)  \right)
\ = \ d\left( \psi'(t)^{-1},  \psi'(t^{-1})    \right)  \ \leq \  2\d'.
\end{equation}

And,  likewise, as  $t^{-1}, b, t, bt, t^{-1} bt \in S'$  and $\phi(\iota(b))  =    t^{-1}\iota(b) t$ in $\H_k(G,\phi)$,
\begin{equation} \label{fourth}
d (\psi'(t^{-1}) \circ \psi' (\iota(b))  \circ   \psi'(t), \psi'(\iota(\phi(b))) )  \ = \
d\left( \psi'(t^{-1})  \circ \psi'(\iota(b)) \circ   \psi'(t), \psi'(t^{-1} \iota(b)    t )
\right) \ \leq \  2\d'.
\end{equation}
In combination, \eqref{first}--\eqref{fourth} yield the first inequality of:
$$
d\left(\psi(b) \circ f, f \circ \psi(\phi(b))\right)  \ = \
d\left(
f^{-1} \circ \psi(b) \circ  f,
 \psi(\phi(b))
\right) \ \leq \  (2k+ 4) \d'  \ \leq \  6k \d'  \ \leq \  \varepsilon.
$$
\end{proof}

The following lemma will provide the heart of our proof that (\ref{main-precise2})~$\thus$~(\ref{main-precise1}).
Since $\H_k(G,\phi)$ is generated by $\iota(G)$ and $t$, we can choose a section  $\sigma: \H_k(G,\phi) \to \set{ G, t }^{\ast}$ for the evaluation map $\set{ G, t }^{\ast} \to \H_k(G,\phi)$---that is, for every  $g \in \H_k(G,\phi)$ we choose  a way of expressing $g$ as a product $\sigma(g) = \iota(g_1) t^{j_1} \cdots \iota(g_r) t^{j_r}$ of elements of $G$ and powers of $t$.

Given a map $\psi : G \to \Sym(n)$ (not necessarily a homomorphism)  and  a  permutation $f \in \Sym(n)$, define a map $\psi^f : \H_k(G,\phi) \to \Sym(n)$  by
$$
 \psi^f (g) \ := \   \psi(g_1)f_i^{j_1} \cdots \psi(g_r)f_i^{j_r},
$$
where $\sigma(g) = \iota(g_1) t^{j_1} \cdots \iota(g_r) t^{j_r}$.  The lemma will tell us that if $\psi$ and $f$ are suitably compatible then $\psi^f$ is close to a homomorphism.

\begin{lemma}
\label{to almost conjugating}
For all finite subsets $S \subseteq \H_k(G, \phi)$  and $\bar S \subseteq G$ such that $\iota(\bar S) \subseteq S$ and all $\d >0$, there exists a finite set
$S_0 \subseteq G$  with $\bar S \subseteq S_0$
and an $\varepsilon >0$ satisfying the following.  Suppose  $\psi: G \to \Sym(n)$ is an
 $(S_0,\varepsilon, n)$-approximation  and $f\in \Sym(n)$ is a permutation  of order dividing $k$ such that for all $b \in S_0 \cap B$
$$
d( \psi(\phi(b)) \circ f, f \circ \psi(b)) \ <  \  \varepsilon.
$$
Then for all $s_1,s_2 \in S$ for which $s_1s_2 \in S$,
\begin{equation}
d\left(   \psi^f(s_1)   \psi^f(s_2),   \psi^f(s_1 s_2)\right) \ < \  \d \label{nearly}
\end{equation}
and for all  $g \in \bar S$
\begin{equation}
d\left(\psi^f(\iota(g)), \psi(g) \right) \ <  \  \d. \label{final ineq}
\end{equation}
\end{lemma}

\begin{proof}
Since $S$ is finite there exists an integer  $m$ and  a finite subset $S' \subseteq G$ containing $\bar  S$ such that $\sigma(S) \subseteq \set{\iota(S'), t}^m$.
Then  $S$ sits inside the  subgroup $\Gamma = \langle \iota(S'), t \rangle$ of $\H_k(G, \phi)$.  As $\Gamma$ is  finitely generated, there exists a finitely presented group
$\Gamma' = \langle S',t \mid R' \rangle$  which projects onto $\Gamma$---that is, the composition $S' \hookrightarrow \Gamma' \onto \Gamma$ is the identity.

By construction, every relation in $R'$ is also satisfied in $\H_k(G,\phi)$, and
so can be deduced from the defining relations in the presentation of $\H_k(G,\phi)$.
These defining relations come in three types: relations in $G$,
the relation $t^k=1$, and relations of the form $\iota(b)^t = \iota(\phi(b))$ for some $b \in B$. We can enlarge the set $S'$ to another finite subset $S'' \subseteq G$ by gathering all elements in $G$ needed to deduce all
the relations $r\in R'$, so as to view $S$ as a subset of a finitely presented group
$$
\Gamma''  \ = \  \langle S'',t \mid t^k, \  R'',  \  b^t \phi(b)^{-1} \mbox{ for } b\in B''   \rangle
$$
where $R''$ is a finite set of relations satisfied in the subgroup $\langle S'' \rangle$ of $G$, and $B''$ is a finite subset of $B$. Let $N  \geq k$ be a number such that every defining relation in $R''$ has length at most $N$ in the generating set $S''$ and all elements in $B''$ and $\phi(B'')$ can be expressed as words in $S''$ of length at most $N -1 $.
By construction, there exists a constant $M$ such that each relator in $\Gamma''$ of the form $s_3^{-1}s_1s_2$ for $s_1, s_2, s_3 \in  S$
or of the form  or $g^{-1} \sigma(\iota(g))$ for $g \in  \bar S$
can be written as product of at most $M$ conjugates of the defining relators in the above presentation.

Define $S_0 = (S'')^N$ and $\varepsilon = \d/ 8MN$.
Suppose that $\psi$ is an $(S_0,\varepsilon, n)$-approximation of $G$ and   $f \in \Sym(n)$ is a permutation of order dividing $k$ such that
$d( \psi(\phi(b)) \circ f, f \circ \psi(b)) < \varepsilon$ for all $b \in S_0 \cap \phi^{-1}(A \cap S)$.
Extend $\psi \restricted{S''}$ to a
homomorphism $\tilde \psi$ from the free group generated by $S''$ and $t$ to $\Sym(n)$ by mapping $t$ to the permutation $f$.
Defining relations $r = t^k$ or $r \in R''$ or $r=b^t\phi(b)^{-1}$  in our presentation of $\Gamma''$ have lengths at most $k$, $N$ and $2N$, respectively, and so $d(\tilde \psi(r),\id) \leq 2N \varepsilon$ by Lemma~\ref{hamming lemma} (i).
It then follows from Lemma~\ref{hamming lemma} (i) and (ii) that for all relators $r'$ of the form  $s_3^{-1}s_1s_2$  or $g^{-1} \sigma(\iota(g))$,
we have
\begin{equation}
d(\tilde\psi(r'),\id)  \ < \  2MN \varepsilon \ < \ \d/4. \label{r closeness}
\end{equation}
For the relators of the first type this gives us that $d(\tilde\psi(s_3^{-1}s_1s_2),\id) \ < \ \d/4$.
For those of  second type we get  both \eqref{final ineq}
for all $g \in \bar{S}$, and
\begin{equation}
d(\tilde\psi(s_i),  \psi^f( s_i)) \ < \ \d/4  \label{s_i closeness}
\end{equation}
 for  $i =1,2,3$.
Then  \eqref{r closeness} applied to $r'=s_3^{-1} s_1 s_2$ and \eqref{s_i closeness} give
\begin{align*}
d(\psi^f(s_3), \psi^f(s_1)\psi^f(s_2)) & \ =  \ d(\psi^f(s_3)^{-1} \psi^f(s_1) \psi^f(s_2),\id) \\
& < \    d(\tilde\psi(s_3)^{-1}\tilde\psi(s_1)\tilde\psi(s_2),\id) + \frac{3\d}{4}  \\
& = \  d(\tilde\psi(s_3^{-1}s_1s_2),\id)  + \frac{3\d}{4} \\
& <  \ \d,
\end{align*}
which yields inequality \eqref{nearly}.
\end{proof}

 \begin{proof}[Proof of Theorem~\ref{main-precise}, (\ref{main-precise2})~$\thus$~(\ref{main-precise1})]
This proof is similar to that  of  \eqref{sof1} $\thus$ \eqref{sof3} of Proposition~\ref{equiv}. Define
$$
I  \ := \ \set{\left. \left(S, \bar{S}, \delta\right) \, \right| \, \text{ finite } S \subseteq \H_k(G, \phi),  \text{ finite } \bar{S} \subseteq G   \text{ with } \iota(\bar{S}) \subseteq S, \ \delta >0 }.
$$
For $\left(S, \bar{S}, \delta
\right) \in I$, define $$\overline{\left(S,  \bar{S}, \delta\right)} \ := \  \set{\left. \left(S',   \bar{S}', \delta'\right) \in I \, \right| \,  S' \supseteq S,  \ \bar{S}' \supseteq \bar{S}, \ \delta' \leq \delta }.$$
As in our proof of Proposition~\ref{equiv}, the family  $\mathcal{F}$ of all subsets $\overline{\left(S, \bar{S}, \delta\right)}$  enjoys the finite intersection property,  and so there is  an ultrafilter $\mathcal{U}$ on $I$ with $\mathcal{F} \subseteq \mathcal{U}$.

Suppose $i = \left(S, \bar{S}, \delta\right) \in I$. Let $S_0 \subseteq G$ and $\ep>0$ be as per Lemma~\ref{to almost conjugating}.  Let  $\psi_i$   be an $(S_0, \ep, n_i)$-approximation of $G$ and   $f_i \in \Sym(n_i)$ a permutation as per  condition~(\ref{main-precise2}).  Together   $\psi_i$ and $f_i$   define maps
$\psi_i^{f_i} : \H_k(G,\phi) \to \Sym(n_i)$ and   Lemma~\ref{to almost conjugating} tells us that these $\psi_i^{f_i}$ enjoy  conditions \eqref{nearly} and \eqref{final ineq}.

If $g \in \bar{S} \subseteq S_0$, then \eqref{final ineq} gives us that  $d\left(\psi_i^{f_i}(\iota(g)), \psi(g) \right)  <   \d$.  If, additionally, $g \neq e$, then $d(\psi_i(g), \id) > 1-\ep$ because $\psi_i$ is an $(S_0, \ep, n_i)$-approximation.  Together these give
\begin{equation}
d(\psi_i^{f_i}(\iota(g)), \id) \ > \ 1-\delta-\ep \label{away from zero}
\end{equation}
for all $g \in \bar{S} \ssm \set{e}$.

The  $\set{\psi^{f_i}_i}_{i \in I}$  combine  to induce  a map
$$
\Psi^{\pmb{f}}: \H_k(G, \phi) \to \prod_{\mathcal U} \Sym(n_i ).
$$
This is a  group homomorphism because of condition \eqref{nearly}.  Its image $Q = \Psi^{\pmb{f}} \left(\H_k(G,\phi)\right)$ is a sofic quotient of $\H_k(G,\phi)$. In general,  $\Psi^{\pmb{f}}$ might not be injective, but
 \eqref{away from zero} tells us that
the composition   $G \stackrel{\iota}{\to} \H_k(G, \phi) \stackrel{\Psi^{\pmb{f}}}{\to}  Q$  is  injective.  In both cases, the details are similar to  our derivations of corresponding statements  in our proof of Proposition~\ref{equiv}.
\end{proof}

\begin{proof}[Proof of Theorem~\ref{main-precise},  (\ref{main-precise3})~$\thus$~(\ref{main-precise2})]
This implication is  immediate since $G$ is sofic.
\end{proof}

The remaining  implication  (\ref{main-precise2})~$\thus$~(\ref{main-precise3}) is significantly more complicated and uses that for an amenable group, any two   approximations into  the same $\Sym(n)$ are almost conjugate. This result is due  to Helfgott and Juschenko in the form given  but, as they explain, has origins in Elek and Szabo~\cite{es2}, builds on a lemma from Kerr and Li~\cite{kl},    and  is also comparable to Arzhantseva and P\v{a}unescu~\cite{ap}.   Helfgott and Juschenko's proof is a delicate analysis of the interplay between sofic approximations and the F{\o}lner characterization of amenability.

\begin{theorem}[Helfgott--Juschenko~\cite{he-ju}]
\label{amenable-conjugate}
Suppose $G$ is an amenable group,   $\ep > 0$, and $S$ is a finite subset of $G$. Then there is a finite subset $S' \subseteq G$ with  $S \subseteq S'$ and constants $N \in \Z^+, \d > 0$ such that for any two $(S',\d,n)$-approximations $\rho_1,\rho_2$ of $G$ with $n \ge N$, there exists $\tau \in \Sym(n)$ such that, for every $s \in S$,
$$
d(\tau^{-1}\circ \rho_1(s) \circ \tau,\rho_2(s)) < \ep.
$$
\end{theorem}

We will also use the following lemma which essentially says that the $n$ in the definition of an $(S,\d,n)$-approximation is irrelevant  provided  it is sufficiently large.

\begin{lemma}
\label{embed}
Suppose $n = qm +r$ where $m,n,q,r$ are non-negative integers with $m,n \geq 1$ and $q = \lfloor n/m \rfloor$.  If  $\alpha: S \to \Sym(m)$ is an $(S, \eta, m)$-approximation of a finite subset $S$ of a group, then  composing
$$
(\overbrace{\a, \ldots, \a}^{q},1): S \to \overbrace{\Sym(m) \times \cdots \times \Sym(m)}^q \times \, \Sym(r)
$$
with the diagonal embedding into  $\Sym(n)$
gives an    $(S, \eta + \dfrac{1}{q+1}, n)$-approximation $\beta$.
\end{lemma}

\begin{proof}
If $s_1, s_2, s_1s_2 \in S$, then $$
d\left(\b(s_1)\b(s_2),\b(s_1s_2)\right)  \ < \   \frac 1n \left(\overbrace{m\eta + \cdots + m\eta}^q\right)  \ = \   \frac {q m\eta}{qm+r}   \ < \
 \eta  \  < \ \eta + \frac{1}{q+1}.
$$

As for the second condition on approximations, suppose $s \in S \ssm \set{e}$. Then
\begin{align*}
d\left(\b(s), \id\right) &> \frac 1n \left(\overbrace{m(1-\eta) + \cdots + m(1-\eta)}^q\right) \ = \   1 - \eta - (1-\eta)\frac rn \ > \  1 - \eta - \frac{1}{q+1},
\end{align*}
with the final inequality coming from combining $1 - \eta <\eta$ and $r/n < 1/(q+1)$, the latter of which holds because $r<m$ implies that $qr+r <qm+r=n$.
\end{proof}

\begin{proof}[Proof of Theorem~\ref{main-precise} (\ref{main-precise2})~$\thus$~(\ref{main-precise3}).]
We are given a finite set  $S \subseteq G$ and some $\varepsilon >0$.

We aim to show that  for a suitable  finite set $S' \subseteq G$ with $S \subseteq S'$ and suitable  $\d >0$   and $N$,  every $(S',\d,n)$-sofic approximation $\psi$ of $G$ with $n> N$  admits some $f \in \Sym(n)$  of order dividing $k$ almost conjugating the action of $A$ under   $\psi$ to the action of $B$ under   $\psi$.
The idea will be to apply   condition~(\ref{main-precise2}) and Lemma~\ref{embed} to obtain some  approximation $\bar{\psi}$ of $G$ together with a permutation $\bar{f}$ of order $k$ which will almost conjugate the action of $A$ to the action of $B$. A priori $\psi$ and $\bar{\psi}$ will be unrelated, but  in fact by Theorem~\ref{amenable-conjugate}  will essentially be conjugate.  We will apply this conjugation   to $\bar{f}$ to obtain the requisite permutation $f$.

Here are the  details.
Let $\tilde \varepsilon =   \varepsilon/3$.  By Theorem~\ref{amenable-conjugate} there exits a finite subset $S' \subseteq G$ with $S \cup \phi(S \cap B ) \subseteq S'$  and $\tilde \d >0$ and $N_0 \in \Z^+$ such that any two $(S',\tilde \d,n)$-approximations $\rho_1$ and $\rho_2$ of $G$ with $n \geq N_0$ are \emph{almost conjugate} in that  there exists $\tau \in \Sym(n)$ such that for all $s \in S \cup \phi(S \cap B)$,
\begin{equation}
d(\tau^{-1}\circ \rho_1(s) \circ \tau, \ \rho_2(s))  \ < \  \tilde \varepsilon.
\end{equation}
Let $\d =\min\{\tilde \delta, \tilde \varepsilon\}$.

By Condition~(\ref{main-precise2}), there exists an $(S', \d/2,m)$-approximation $\psi'$  of $G$ together with a permutation $f \in \Sym(m)$  of order dividing $k$ such that for all $b \in S'  \cap \phi^{-1}(A \cap S')$,
\begin{equation}
d( \psi(b) \circ f, \ f \circ \psi(\phi(b)))  \ < \   \tilde \varepsilon.
\end{equation}

Let  $N = \max\{N_0,2m / \d\}$.  With $S'$ and $\d$  as defined above, suppose $\psi$ is an $(S',\d,n)$-approximation of $G$ with $n > N$.

Via Lemma~\ref{embed}, we can use $\psi'$ to construct another
$(S', \d,n)$-approximation $\bar \psi$ of $G$ and an associated permutation $\bar f$ which almost conjugates the action of $A$ to $B$ with the same error $\tilde\varepsilon$:
\begin{equation}
d( \bar \psi(b) \circ \bar f,  \ \bar f \circ  \bar \psi(\phi(b)))  \ < \   \tilde \varepsilon \label{dve}
\end{equation}
for all $b \in S' \cap B$.
This is possible because
$$
\frac{\d}{2} + \frac{1}{\lfloor  n/m\rfloor +1 }  \ \leq  \ \frac{\d}{2} +  \frac{1}{\lfloor  2 / \d\rfloor +1 } \  < \
\frac{\d}{2} + \frac{\d}{2}  \ \leq \  \d
$$
and given how  $\bar{\psi}$ is assembled from copies of $\psi$ and the identity (and correspondingly $\bar{f}$ from copies of $f$ and the identity), the error $\tilde \varepsilon$ of \eqref{dve} does not increase and $\bar{f}$, like $f$, has order dividing $k$.

By Theorem~\ref{amenable-conjugate} there is a permutation $\tau \in \Sym(n)$ which almost conjugates
$\psi$ to $\bar \psi$---i.e.,
\begin{equation}
d\left( \tau^{-1} \circ  \psi(s) \circ \tau, \  \bar \psi(s)\right)  \ < \   \tilde \varepsilon \label{edno}
\end{equation}
for all $s \in S'$.

Define $f  = \tau \circ  \bar f \circ \tau^{-1}$,  which is a permutation of order dividing $k$ since $\bar f$ has order dividing $k$.  Suppose  $b \in S  \cap \phi^{-1}(A \cap S)$.  We will complete our proof by showing that
$$
d(\psi(b) \circ f, \ f \circ \psi(\phi(b)) \   < \  \varepsilon.
$$
By definition of $f$,
\begin{equation}
f^{-1} \circ \psi(b) \circ  f  \ = \  \tau \circ \bar f^{-1} \circ  \tau^{-1} \circ \psi(b) \circ  \tau   \circ   \bar f \circ \tau^{-1}. \label{shest}
\end{equation}
Since $b \in S \subseteq S'$, by \eqref{edno},
\begin{equation}
d\left(
\tau \circ \bar f^{-1} \circ \tau^{-1} \circ \psi(b) \circ  \tau  \circ  \bar f \circ \tau^{-1}
,
\tau \circ \bar f^{-1} \circ \bar \psi (b)  \circ   \bar f \circ \tau^{-1}
\right) \ < \  \tilde  \varepsilon. \label{tri}
\end{equation}
By \eqref{dve},
$
d \left(\bar f^{-1} \circ \bar\psi (b)  \circ  \bar  f,  \ \bar\psi(\phi(b)) \right)  <  \tilde \varepsilon,
$
and therefore
\begin{equation}
d\left(
\tau \circ \bar{f}^{-1} \circ \bar\psi (b)  \circ   \bar{f} \circ \tau^{-1}
,
\tau \circ \bar\psi (\phi(b ))  \circ \tau^{-1}
\right) \ <  \ \tilde  \varepsilon. \label{chetiri}
\end{equation}
Since $\phi(b) \in S \cap B \subseteq S'$, by \eqref{edno} again,
\begin{equation}
d\left(
\tau \circ \bar\psi (\phi(b) )  \circ \tau^{-1}
,
\psi(\phi(b))
\right) \ < \  \tilde  \varepsilon. \label{pet}
\end{equation}

Together,  \eqref{shest}--\eqref{pet} yield the first inequality of:
$$
d(\psi(b) \circ f, \ f \circ \psi(\phi(b))  \ = \ d\left(
f^{-1} \circ \psi(b) \circ f,
\psi(\phi(b))
\right) \ < \  3 \tilde \varepsilon \  =  \ \varepsilon,
$$
which completes the proof.
\end{proof}

\section{Applications of Theorem~\ref{main-precise}}
\label{examples}
\label{main proof section}

In this section we will examine the groups $\Z^2$, the 3-dimensional integral Heisenberg group $\mathcal{H}$, $\BS(1,m)$,  $\Z \wr \Z$, and the 2-generator metabelian group  in the context of Theorem~\ref{main-vague} (or, in its precise form, Theorem~\ref{main-precise}).  Each of these groups is amenable.  We will exhibit families of maps witnessing to their soficity, and will then  explain what Theorem~\ref{main-precise} allows us to conclude about the existence of seemingly pathological permutations.  In particular, we will explain how the case of $\BS(1,m)$ yields Theorems~\ref{HJ thm} and  \ref{GS version}, and how  $\Z \wr \Z$ yields Theorem~\ref{heuristicapplication}.

We begin with $\Z^2$, which we view as an introductory example.

\subsection{\texorpdfstring{$\Z^2$}{Z x Z}}
\label{ZbyZ section}  We present $\Z^2$ as  $\langle a, b \mid ab=ba \rangle$, so  $\phi : \Z \to \Z$, given by $b \mapsto a$, is the map defining  $\H_k(\Z^2)$.

To obtain a  family of functions witnessing to the soficity of $\Z^2$, we identify  $\Sym(n)$ with $\Sym(\Z/n\Z)$ and then for $p, q \in \N$, define $\psi_{n,p,q}: G \to \Sym(n)$ by
\begin{align*}
\psi_{n,p,q}(a) & : x \mapsto x+p \text{, and} \\
\psi_{n,p,q}(b) & : x \mapsto x+q.
\end{align*}

\begin{lemma}
\label{lm-soficapproximationZbyZ}
For any finite set $S \subseteq \Z^2$ and any $\d >0 $, there exists a constant $C$ such that $\psi_{n,p,q}$ is an $(S,\d,n)$-approximation of $\Z^2$ provided that $p> Cq$ and $n > Cp$.
\end{lemma}
\begin{proof}
Take $C$ sufficiently large that  $S \subseteq  \left\{ \left. a^{\lambda}b^{\mu} \ \right| \  |\lambda| < C/3, \ |\mu| < C/3\right\}$.
Since the map $\psi_{n,p,q}$ is a group homomorphism, we only need to show that
$d(\psi_{n,p,q}(s), \id) > 1 - \d$ for all $s \in S\smallsetminus\set{1}$ provided that $p> Cq$ and $n > Cp$.
Then for $s =  a^{\lambda}b^{\mu} \in S$ we find $\psi_{n,p,q}\left(s\right)$ is translation by $\lambda p + \mu q$, which is not divisible by $n$ (unless $\lambda=\mu=0$), and therefore $d(\id, \psi_{n,p,q}(s)) = 1$.
\end{proof}

The equivalence \emph{(\ref{main-precise1})}~$\Leftrightarrow$~\emph{(\ref{main-precise3})} of Theorem~\ref{main-precise}    tells us that for  $k \in \N$, the group $\H_k(\Z^2)$ has a sofic quotient $Q$ such that the composition $\Z^2 \to \H_k(\Z^2) \to Q$ is injective if and only if for any $n,p,q$ such that $n/p$ and $p/q$ are sufficiently large, there is a permutation $f: \Z/n\Z \to \Z/n\Z$ of order $k$ with $d( \psi_{n,p,q}(b) \circ f, f \circ \psi_{n,p,q}(a))   <      \varepsilon$.  As   $(\psi_{n,p,q}(b) \circ f)(x) = f(x) +q$ and    $(f \circ \psi_{n,p,q}(a))(x)  = f(x+p)$, the latter condition amounts to  $f(x+p) = f(x) + q$ for at least $(1-\ep)n$ elements  $x \in \Z/n\Z$.

However, for    $k \geq 1$, the group  $\overline\H_k(\Z^2)$  is a right-angled Artin group, so it is linear and thus residually finite (see~\cite{hs-wi}).  Thus $\H_k(\Z^2)$ and  $\overline{\H}_k(\Z^2)$ are sofic.   (For $k\geq 4$,  we reached the same conclusion  in Corollary~\ref{so sofic cor} via the residual solvability established in Theorem~\ref{ressolvable}.  For $k \leq 3$ the group is abelian, and thus also sofic.) And for $k \geq 2$, $\Z^2 \hookrightarrow \overline\H_k(\Z^2)$.
Thus:

\begin{theorem}
\label{th-functions-ZbyZ}
Suppose $k \geq 2$ and $\ep >0$.  Then there exists $C>0$ such that  for all $n,p,q$ satisfying $n \geq Cp$ and $p\geq Cq$,  there exists a permutation $f \in \Sym(\Z/n\Z)$  of order dividing $k$ such that
$$
f(x+p) = f(x) + q
$$
for at least $(1-\ep)n$ elements  $x \in \Z/n\Z$.
\end{theorem}

In some cases this is  straight-forward. If $n$ is a prime congruent to $1$ modulo $k$, then there exists  $l \in \Z/n\Z$ such that $l^k = 1$  and $q = lp$ in $\Z/n\Z$, and then $f: \Z/n\Z \to \Z/n\Z$ mapping $ x \mapsto lx$ satisfies the given conditions, because $f(x+p)   = lx + lp   = lx + q = f(x) +q$ for all $x \in \Z/n\Z$ and $f^k = \id$.  Indeed, such $f$ arise  from a natural sofic quotient of $\H_k(\Z^2)$---take the semidirect product of the cyclic group of order $k$ and the abelianization of $\overline \H_k(G)$.  Then $\H_k(\Z^2)$ maps onto $C_k \ltimes \Z/n\Z$, where the action is by multiplication by $l$.

But in most cases errors are inevitable.  Suppose $f:  \Z/n\Z \to \Z/n\Z$ satisfies $f(x+p) = f(x) + q$ for all $x \in  \Z/n\Z$.  Then $f^l(x+q^l) = f^l(x) + q^l$ for all $x \in \Z/n\Z$ and all $l \in \N$.  So if $f^k = \id$, then  $n$ divides $q^k-p^k$.

Whether or not $n$ divides $q^k-p^k$, by  Theorem~\ref{th-functions-ZbyZ}, there exist such functions $f$ satisfying $f(x+p) = f(x) + q$ for most $x \in \Z/n\Z$. Such $f$ could be constructed explicitly by carefully following the arguments in our proofs of Theorems~\ref{main-precise} and \ref{amenable-conjugate}  (using that there are F{\o}lner sets for $\Z^2$ of a very simple form), but doing this in general would be quite technical.

\begin{remark} \label{dependance1}
In such a simple example it is possible to determine the dependance of the constant $C$ in Theorem~\ref{th-functions-ZbyZ}: one can   take $C = O(\ep^{-k})$.
\end{remark}

\subsection{The Heisenberg group}
\label{Heis section}
The Heisenberg group $\Heis$ has presentation
$$
\Heis = \langle a, b \mid [a,[a,b]] = [b,[a,b]] = 1 \rangle.
$$
It is nilpotent and so  is amenable and residually  finite.  Identify $\Sym(n^2)$ with $\Sym\left((\Z/n\Z)^2\right)$. Define $\psi_n: \Heis \to \Sym(n^2)$  for  $n \in \N$ by
\begin{align*}
\psi_n(a) & : (x,y) \mapsto (x,y+1) \text{, and} \\
\psi_n(b) & : (x,y) \mapsto (x+y,y),
\end{align*}
which extends to  $\Heis$ since $\psi_n(a)$ and $\psi_n(b)$ satisfy the defining relations of $\Heis$. This action of $\Heis$ arises the finite quotient $\Heis_n := \Heis/ \langle a^n,b^n \rangle$  acting on cosets of the subgroup
$\langle a \rangle$.
\begin{lemma}
\label{lm-soficapproximationHeis}
For all finite sets $S \subseteq \Heis$ and all $\d >0 $, there exists $C>0$ such that  $\psi_{n}$ is $(S,\d,n^2)$-approximation of $\Heis$  for all $n> C$.
\end{lemma}
\begin{proof}
Suppose $S \subseteq \mathcal{H}$ is finite.
As in Lemma~\ref{lm-soficapproximationZbyZ}, as $\psi_n$ is a homomorphism, it suffices to check that  there exists an integer $C$ such that when $n > C$,  the permutation  $\psi_{n}(s)$ is far from the identity for all $s \in  S \ssm \set{e}$.
Every element of $\Heis$ can be expressed uniquely as $a^{\lambda}b^{\mu}[a,b]^{\nu}$.
So there exists $N$ such that
$$
S \subseteq \left\{ \left. a^{\lambda}b^{\mu}[a,b]^{\nu} \, \right| \, |\lambda|, |\mu|, |\nu| < N \right\}.
$$
One computes that
$$
\psi_n(a^{\lambda}b^{\mu}[a,b]^{\nu}): (x,y) \mapsto (x+ \mu y - \nu,y + \lambda),
$$
which is a permutation with at most $|\lambda| n$ fixed points  (provided that this element is not trivial---i.e., $n \nmid \lambda$ or $n  \nmid  \mu$ or $n \nmid   \nu$). Therefore if $n > C := N/\d$, then $\psi_{n}$ is an $(S,\d,n^2)$-approximation.
\end{proof}

So Theorem~\ref{main-precise}  tells us that for all $k \in \N$ the group $\H_k(\Heis)$ has a   sofic quotient $Q$ such that $\mathcal{H} \to \H_k(\Heis) \to Q$ is injective if and only there exist infinitely many  $n$ and functions $f: (\Z/n\Z)^2 \to (\Z/n\Z)^2$ of order dividing $k$ which conjugate the action of $b$ to the action of $a$ under $\psi_n$ up to error $\ep$.

For $k\geq 4$, Corollary~\ref{so sofic cor} tells us that $\H_k(\Heis)$ is sofic. As for the case $k=2$, we have that $\bar \H_2(\Heis)\simeq \Heis$ is also sofic.   And for $k=3$, it is not hard to construct a  surjective map from $\overline\H_3(\Heis)$ onto the sofic group $\mathrm{SL}_3(\Z)$  such that the composition $\Heis \to \bar \H_3(\Heis)\to \mathrm{SL}_3(\Z)$ is injective.
  A sofic quotient of $\H_3(\Heis)$ into which $\Heis$ injects could also be obtained via Proposition~\ref{soficquotient}.
(We do not know whether the group $\H_3(\Heis)$ itself is sofic, but   see no reason it should not be.)

So applying the condition
$d( \psi(b) \circ f, f \circ \psi(\phi(b)))    <      \varepsilon$
of Theorem~\ref{main-precise}  \eqref{main-precise3}   to  $\psi_n$ and to $b$ and $a=\phi(b)$ we get the following.

\begin{theorem}
\label{th-functions-Heis}
For all $\ep >0$ and all $k \geq 2$, there exists an integer $C$ such that for all $n > C$ there exists a
permutation $f \in \Sym(n^2)$ of order dividing $k$ which, when expressed as  $f(x,y) = \left(f_{1}(x,y), f_{2}(x,y) \right)$ so that $f_{1},f_{2}: \Z/n\Z \times \Z/n\Z \to \Z/n\Z$ are its coordinate functions, satisfies
$$
f_{1}(x,y+1) = f_{1}(x,y) + f_{2}(x,y)
\quad \text{ and } \quad
f_{2}(x,y+1) = f_{2}(x,y)
$$
for at least $(1-\ep)n^2$ pairs $(x,y)$.
\end{theorem}

\begin{remark} \label{dependance2}
There is no $f$ such that the above equality holds for all pairs $(x,y)$, since $\psi_{n}(a)$ and $\psi_{n}(b)$ are not
conjugate inside $\Sym(n^2)$---one of them has (a few) fixed points and the other has none.
\end{remark}

\begin{remark}
It is again possible to estimate the dependance of the constant $C$. One can show that we can take $C = O(\ep^{-3k})$.
\end{remark}

\begin{remark}
The generators $a$ and $b$ play asymmetric roles in the definition of $\psi_n$: for example,
$d(\psi_n(a),\id ) = 1$,
but
$d(\psi_n(b),\id ) =1-  1/n <1$.
One can instead take the action of $\Heis_n$ on itself which will lead to a permutation representation $\Heis \to \Sym(n^3)$ in which the roles of $a$ and $b$ are symmetric.   The functions $f$ of the resulting analogue of Theorem~\ref{th-functions-Heis} can be constructed in such a way that the equations are satisfied for \emph{all} points if and only if there is an nontrivial semisimple element of order dividing $k$ in the group $\SL_2(\Z/n\Z)$.
\end{remark}

\begin{remark}
One can also consider a   Higman-like construction from $\Heis$ in which $\phi$ maps one of the standard generators to a generator of the center: define 
$$
G = \Heis = \langle b, c \mid [b,[b,c]] = [c,[b,c]] =1 \rangle, \quad  \phi: b \mapsto a,
$$
where $a := [b,c]$ so that
$$
\overline\H_k(G, \phi) = \langle \, b_1, c_1,  \dots, b_k, c_k \mid  [b_i, [b_i, c_i]]= [c_i, [b_i, c_i]] = 1,  b_{i} = [b_{i+1},c_{i+1}]  \ \forall i \ (\textup{mod } k) \rangle.
$$

We do not know whether the group $\overline\H_k(G, \phi)$ is sofic, since Theorem~\ref{ressolvable} does not apply.  However $b_i \mapsto \id + e_{i,k+1}$ and $c_i \mapsto \id - e_{i+1,i}$ (indices mod $k$) defines a homomorphism
$$
\overline\H_k(G, \phi) \  \to \  \SL_k(\Z) \ltimes \Z^k \ \subseteq  \ \SL_{k+1}(\Z),
$$
and for $k\geq 2$ the group $\Heis$ injects into this quotient (i.e.,   image) of $\overline\H_k(G, \phi)$ which is linear and thus  sofic.

As before (but with $a$ and $b$ now changed to $b$ and $c$, respectively) define $\psi_n: \Heis \to \Sym(n^2)$  for  $n \in \N$ by
\begin{align*}
\psi_n(b) & : (x,y) \mapsto (x,y+1) \text{, and} \\
\psi_n(c) & : (x,y) \mapsto (x+y,y).
\end{align*}
Then, as $\psi_n([b,c]) : (x,y) \mapsto (x-1,y)$, applying Theorem~\ref{main-precise} leads to:
\end{remark}

\begin{theorem}
\label{th-functions-Heis2}
Functions $f$ exist exactly as per Theorem~\ref{th-functions-Heis}, except with the displayed equations replaced by:
$$
f_{1}(x-1,y) = f_{1}(x,y) + f_{2}(x,y)
\quad \text{ and } \quad
f_{2}(x-1,y) = f_{2}(x,y).
$$
\end{theorem}

Despite their similarity, we do not see a way to derive  one of  Theorems~\ref{th-functions-Heis} and  \ref{th-functions-Heis2}  immediately from the other.  Defining $g(x,y): = f(y,x)$  transforms one set of recurrences to the other, but the condition that the function's order divides $k$ is lost.

\subsection{The Baumslag--Solitar group \texorpdfstring{$\BS(1,m)$}{BS(1,m)}}
\label{BS section}
This is the case addressed by Helfgott and Juschenko in~\cite{he-ju}.  Here we explain how it fits into our framework and give our own account of how it relates to recent work of Glebsky.

The Baumslag--Solitar group $\BS(1,m)$ has  presentation
$$
\BS(1,m)   \ = \  \langle a, b \mid a^b=a^m \rangle.
$$
It is a residually finite solvable group, and so is amenable.
If $m \not= \pm 1$, then the image of $a$ in any proper quotient of $\BS(1,m)$ is finite.
(Every element can be expressed as $b^{\mu} a^{\nu} b^{-\lambda}$ for some $\mu,   \lambda \geq 0$ and $\nu \in \Z$.  The result then follows from consequences of  a non-trivial $b^{\mu} a^{\nu} b^{-\lambda}$ mapping to the identity and the relation $a^b=a^m$.)

Identify  $\Sym(n)$ with $\Sym(\Z/n\Z)$.
For all $n \in \N$ relatively prime to $m$, define a map $\psi_n: \BS(1,m) \to \Sym(n)$  by
\begin{align*}
\psi_n(a) & : \    x \mapsto x+1 \text{, \  and} \\
\psi_n(b) & :   \  x \mapsto m^{-1}x,
\end{align*}
which extends to a homomorphism
defined on the whole of $\BS(1,m)$ since  $\psi_n(a)$ and $\psi_n(b)$ satisfy the defining relation of $\BS(1,m)$. This action of $\BS(1,m)$ arises from the quotient $\BS(1,m)_n := \BS(1,m)/ \langle a^n \rangle$ acting on cosets of the subgroup $\langle b \rangle$.
\begin{lemma}
\label{lm-soficapproximationBS}
For all finite sets $S \subseteq \BS(1,m)$ and all $\d >0$, there exists an integer $C$ such that for all $n > C$,
the map $\psi_{n}$ is an $(S,\d,n)$-approximation of $\BS(1,m)$, provided that $|m| \geq 2$.
\end{lemma}
\begin{proof}
The group $\BS(1,m)$ can be represented by $2 \by 2$ matrices via
$$
a \mapsto \begin{bmatrix}1 & 1 \\ 0 & 1 \end{bmatrix},
\qquad
b \mapsto \begin{bmatrix}1 & 0 \\ 0 & m \end{bmatrix}.
$$
The image of this embedding is
$$\set{ \left. \  \begin{bmatrix}1 & \lambda \\ 0 & m^{\mu} \end{bmatrix} \  \right|  \ \mu \in \Z, \ \lambda \in \Z\left[\frac{1}{m}\right] \  }.$$
For every finite set $S \subseteq \BS(1,m)$ there exists a positive integer $N$ such that every $s\in S$ is sent to
$$
\begin{bmatrix}1 & \lambda_s m^{-N} \\ 0 & m^{\mu_s} \end{bmatrix}
$$
where $\lambda_s$ and $\mu_s$ are integers such that $|\mu_s| \leq N$ and $|\lambda_s| \leq \left| m^{2N} \right|$.
One computes that
$$
\psi_n(s): \left(x \mapsto m^{-\mu_s}(x + \lambda_s m^{-N}) \right).
$$
If this permutation is non-trivial (i.e., $\mu_s \not= 0$ or $n \nmid \lambda_s$), then it has at most $m^N$ fixed points. Therefore if
$n > C := \max\left\{\left| \frac{m^{N}}{\d} \right| , \left| m^{2N} \right|  \right\}$, then $\psi_{n}$ is an $(S,\d,n)$-approximation.
\end{proof}

Theorem~\ref{main-precise}  now tells us that for all $k \in \N$ the group $\H_k(\BS(1,m))$ has a   sofic quotient into which $\BS(1,m)$ naturally embeds if and only there exist infinitely many  $n$  and functions $f: \Z/n\Z \to \Z/n\Z$ of order dividing $k$ which conjugate  addition  to multiplication by $m$ up to error $\ep$---that is,
 $m^{-1} f(x) =  f(x+ 1)$
for   at least $(1-\ep)n$ values of $x \in  \Z/n\Z$.    Define $\tilde f \in \Sym(\Z / n \Z)$ by $\tilde f(x) =- f(-x)$, which  has order dividing $k$ if and only if $f$ does.   The  equation $m^{-1} f(x) =  f(x+ 1)$ can be re-expressed as   $f(x) = m  f(x+1)$,  and then as
$$
\tilde f (x+1) = - f (-x-1) = - m f (-x) = m \tilde f (x).
$$
So   $m^{-1} f(x) =  f(x+ 1)$ is satisfied by at least $(1-\ep)n$ values of $x \in  \Z/n\Z$ if and only if the same is true of $\tilde f(x+1) = m \tilde f(x)$.    Theorem~\ref{HJ thm} then follows, or, in more detail, we have:

\begin{theorem}[Helfgott--Juschenko~\cite{he-ju}]
\label{th-functions-BS}
The group $\H_k(\BS(1,m))$ has a  sofic quotient $Q$ such that the composition $\BS(1,m) \to \H_k(\BS(1,m)) \to Q$ is injective if and only if
for all $\ep >0$ there exists an integer $C$ such that for all $n>C$ coprime to $m$, there exist a permutation $f \in \Sym(\Z / n \Z)$ of order dividing $k$
such that
$
f(x+1) = m f(x)
$
for at least $(1-\ep)n$ values of $x$.
\end{theorem}

We do not know whether  $\H_k(\BS(1,m))$ is sofic for $k\geq 4$.
Both $\H_k(\BS(1,m))$ and $\overline \H_k(\BS(1,m))$ are finite for $k \leq 3$ (assuming $m \not = \pm 1$),
and so cannot have a   sofic quotient into which $\BS(1,m)$ injects.
A beautiful argument due to Higman~\cite{hi} shows that $\overline\H_k(\BS(1,2))$  has no finite quotients.
Glebsky~\cite{gl,gl1} shows that, by contrast, if $k\geq 4$ and $p$ is a  prime dividing $m-1$, then the groups $\overline\H_k(\BS(1,m))$   have many quotients which are finite $p$-groups.   His main theorem in~\cite{gl} amounts to the following.  His proof is more combinatorial than the one we give below via  Golod--Shafarevich machinery.
\begin{theorem}[Glebsky~\cite{gl}]
\label{glebsky thm}
Suppose $k\geq 4$ and $p$ is a prime dividing $m-1$. Then the pro-$p$ completion of $\overline\H_k(\BS(1,m))$ is   infinite.
If, moreover, $k$ is even and $m  \not = \pm 1$, then $\BS(1,m)$ embeds into this pro-$p$ completion.
\end{theorem}
\begin{proof}
The defining relator of $\BS(1,m)$ can be written in the from $[a,b]^{-1} a^{m-1}$
and lies in the $p$-Frattini subgroup of the free group. This implies that the pro-$p$ completion $\hat G$ of $\overline\H_k(\BS(1,m))$ has a minimal pro-$p$ presentation with
$k$ generators and $k$ relations. Such a presentation satisfies the Golod--Shafarevich condition (since $k \leq k^2/4$) and therefore
it defines an infinite pro-$p$ group  (see, for example, \cite{e}).

Since $a$ has finite order in any proper quotient of $\BS(1,m)$, to prove that $\BS(1,m)$ embeds into this pro-$p$ completion, it suffices to show that the images of the generators
$a_1, \ldots, a_k$ of $\overline\H_k(\BS(1,m))$ have infinite order---indeed, that and one of them has infinite order.
In the case $k=4$, the defining relations of $\hat{G}$ are similar to the relations of $F_2 \times F_2$, where the first copy of the free group $F_2$ is generated by
$a_1$ and $a_3$ and the second copy is generated by $a_2$ and $a_4$. It can be shown that   $\hat{G}$ contains the free pro-$p$ groups $\Gamma_1$ and $\Gamma_2$ generated by $\Gamma_1=\langle a_1,a_3\rangle$ and $\Gamma_2=\langle a_2,a_4\rangle$, and  moreover that any element in $\hat G$ can be written uniquely as a product of two elements, one from $\Gamma_1$ and one from $\Gamma_2$. This shows that the order of $a_i$ in $\hat G $ is infinite and that $\BS(1,m)$ embeds in $\hat G$.
(When $k >4$ the group $\hat G$ does not have such nice combinatorial description, but there is a quotient of $\hat G$, which has similar structure, provided that $k$ is even.)
\end{proof}

Theorems~\ref{th-functions-BS} and~\ref{glebsky thm} together imply Theorem~\ref{GS version}:
if $|m|>2$ and   $\ep >0$, then there exists $C$ such that for all $n>C$ coprime to $m$, there are permutations $g \in \Sym(\Z /n\Z)$ with $g^4 = \id$ and with $g(x+1) = m g(x)$ for at least $(1-\ep)n$ values of $x \in \Z /n\Z$.

\begin{remark}
Theorem~\ref{GS version} applies to  all   integers $m \neq 0, 2$, these being the cases where $(m-1) \neq \pm 1 $ and has no prime divisors are $m=0,2$.
The above proof via Theorem~\ref{glebsky thm} works for $m\not=-1,0,1,2$.
For  $m=1$, when $\BS(1,1)=\Z \times \Z$, the analogue is a special case of Theorem~\ref{th-functions-ZbyZ}.
The case $m=-1$ only requires a minor strengthening of Theorem~\ref{glebsky thm}.  The case $m=0$ is degenerate since $a$ and $b$ have different orders.
Also one can  have the order of $g$ divide any given even integer $k \geq 4$, not just $4$.
We stress that the analogue of Theorem~\ref{GS version} is unknown when $m=2$.
\end{remark}
\begin{remark} \label{dependance3}
Estimating the dependance of the constant $C$ in Theorem~\ref{GS version} is quite hard because it involves explicitly constructing the sets $S'$ in Theorem~\ref{amenable-conjugate}. We believe that by carefully tracking all bounds one gets that $C = O\left(2^{K \ep^{-2}}\right)$.
\end{remark}

\subsection{\texorpdfstring{$\Z \wr \Z$}{Z wr Z}}
\label{ZwrZ section}

The wreath product $\Z \wr \Z$  has  presentation
$$
\Z \wr \Z = \left\langle a, b \  \left| \  \left[a,a^{b^i}\right] = 1 \ \forall i \in \N \right. \right\rangle.
$$
It is a residually finite solvable group, and so  is amenable.
For any $n \in \N$ and any $m$  coprime to $n$, define a homomorphism $\psi_{n,m}: \Z \wr \Z \to \Sym(\Z/n\Z)$ by
\begin{align*}
\psi_{n,m}(a) &= \left(x \mapsto x+1\right) \text{, and} \\
\psi_{n,m}(b) &= \left(x \mapsto m^{-1}x\right),
\end{align*}
which is well-defined since the permutations $\psi_{n,m}(a)$ and $\psi_{n,m}(b)$ satisfy the defining relations of $\Z \wr \Z$. This action of $\Z \wr \Z$ arises from the quotient $(\Z \wr \Z)_{n,m} = (\Z \wr \Z) / \langle a^n=1, a^b=a^m \rangle$ acting on  cosets of the subgroup $\langle b \rangle$.
\begin{lemma}
\label{lm-soficapproximationZwrZ}
For all finite sets $S \subseteq \Z \wr \Z$ and all $\d >0$, there exists $C>0$ with the property that
$\psi_{n,m}$ is an $(S,\d,n)$-approximation of $\Z \wr \Z$ for all coprime $n,m$    satisfying $n \nmid t(m)$  for every nonzero polynomial $t(x) = \sum t_i x^i \in \Z[x]$ whose degree is most $C$ and whose coefficients all satisfy $|t_i| < C$.

If, moreover,  $ |m| > 2C+1$ and $n >  |m|^{C +1}$, then all such polynomials satisfy $n \nmid t(m)$, and so $\psi_{n,m}$   is necessarily an $(S,\d,n)$-approximation of $\Z \wr \Z$.
\end{lemma}
\begin{proof}
The group $\Z \wr \Z$ can be represented by the group of  matrices
$$\left\{  \left. \  \begin{bmatrix}1 & \tilde t(x) \\ 0 & x^k \end{bmatrix} \  \right| \  k \in \Z, \ \  \tilde t(x) \in \Z[x,x^{-1}] \  \right\}$$
via
$$
a \mapsto \begin{bmatrix}1 & 1 \\ 0 & 1 \end{bmatrix},
\qquad
b \mapsto \begin{bmatrix}1 & 0 \\ 0 & x^{-1} \end{bmatrix}.
$$
Suppose $S$ is a finite subset of $\Z \wr \Z$.     Then   there exists an integer $N$ with the following property.
Every $s \in S$  can be represented by $$
\begin{bmatrix}1 & \tilde t_{(s)}(x) \\ 0 & x^{\mu_{(s)}} \end{bmatrix}
$$
where
$$
\tilde t_{(s)}(x) \  = \  \sum_{i=-N}^N \tilde t_i x^i  \ \in \  \Z[x,x^{-1}]
$$
and  $|\mu_{(s)}| \leq N$ and  $|\tilde t_i| \leq N$ for all $i$. One computes that for $s$ as above,
$$
\psi_{n,m}(s): \left( x \mapsto m^{-\mu_{(s)}}\left(x + \tilde t_{(s)}(m) \rule[2ex]{0pt}{0pt}\right) \rule[3ex]{0pt}{0pt} \right).
$$

Define $t_{(s)}(x) :=  x^{N} \tilde t_{(s)}(x) \in \Z[x]$.  Let $C = \max \left\{ 2N+1, \left\lfloor\frac{1}{\d} +1 \right\rfloor \right\}$, so that  $t_{(s)}(x)$ is within the scope of the lemma.
 Assume that $n,m$ satisfy the condition in the first part of the lemma, so that, in particular, $n \nmid t_{(s)}(m)$.

If $\mu_{(s)} =0$,  then  $\psi_{n,m}(s)$ maps $x$ to $x + \tilde t_{(s)}(m)   =  x + m^{-N} t_{(s)}(m)$ and so has no fixed points as $m^{-1}$ is coprime to $n$ and $n \nmid t_{(s)}(m)$.

Suppose $\mu_{(s)} \not =0$. If $n | (m^{\mu_{(s)}} -1)$, then $\psi_{n,m}(s):  x \mapsto x + \tilde t_{(s)}(m)$ and so is either the identity (when $n \mid \tilde t_{(s)}(m)$, which would contradict  $n \nmid t_{(s)}(m)$, and so does not occur) or has no fixed points.  If $n \nmid (m^{\mu_{(s)}} -1)$, then the permutation $\psi_{n,m}(s)$ has at most $\gcd(m^{\mu_{(s)}} -1,n)$ fixed points.  But $\gcd(m^{\mu_{(s)}} -1, n) \leq \d n$, else the polynomial $M(x^{|\mu_{(s)}|} - 1)$ will have $m$ as a root mod $n$ for some $M < 1/\d$.  So $\psi_{n,m}(s)$ has at most $\d n$  fixed points.

In every case  we have $d(\psi_{n,m}(s),\id) > 1 -\d$. And the \emph{almost homomorphism} condition is immediate since $\psi_{n,m}$ is, in fact, a homomorphism.  So  $\psi_{n,m}$
is an $(S,\d,n)$-approximation of $\Z \wr \Z$, as required.

For the final part of the lemma, assume  $|m| > 2C+1$ and $n >   |m|^{C +1}$.  In particular, $\abs{m} \neq 1$.
Let $d \leq C$ be the degree of $t(x)$.
 Then  $$\abs{t(m)} \ \leq \  \sum_{i=0}^d \abs{t_i} \abs{m}^i  \ \leq \  C \frac{\abs{m}^{d+1} -1}{ \abs{m} -1}  \ < \   |m|^{C +1} \ < \ n.$$
If $d=0$, then $0 \neq \abs{t(m)} = \abs{t_0} < C < n$, so   $n \nmid t(m)$, as required.
Assume $d \geq 1$.  Then
$$
\abs{t(m)}  \ = \  \abs{t_d m^d  +   \sum_{i=0}^{d-1} t_i m^i   }  \ \geq \  \abs{t_d m^d} - \abs{ \sum_{i=0}^{d-1} t_i m^i}
 \ \geq \  \abs{m}^d - C \frac{\abs{m}^d -1}{\abs{m} -1}  \ \geq \ |m|^d /2,
$$
where the final inequality holds because $\abs{m} > 2C+1$.
As $d \geq 1$, we now have that  $0 < |m|/2  \ \leq \  |t(m)| < \ n$, which implies that $n \nmid t(m)$.
\end{proof}

So Theorem~\ref{main-precise}  tells us that for all $k \in \N$ the group $\H_k(\Z \wr \Z)$ has a sofic quotient into which $\Z \wr \Z$ naturally embeds if and only there exist infinitely many
$n$  and functions $f: \Z/n\Z \to \Z/n\Z$ of order dividing $k$ which conjugate addition of $1$ to multiplication by $m$ up to error $\ep$. After conjugating by a minus sign (in the manner of replacing $f$ by $\tilde f$ in Section~\ref{BS section}) this latter condition becomes $f(x+1) = m f(x)$ for at least $(1-\ep)n$ values of $x \in \Z /n\Z$.  However, for $k\geq 4$ we know  by Lemma~\ref{injective}  that  $\Z \wr \Z$ naturally embeds in $\H_k(\Z\wr \Z)$ and  by Corollary~\ref{so sofic cor} that $\H_k(\Z\wr \Z)$  is sofic. We also know by  Proposition~\ref{soficquotient}  that $\H_3(\Z\wr \Z)$ has a sofic quotient into which $\Z \wr \Z$  naturally embeds.
So we have that such $n$ and $f$ do exist for $k \geq 3$.  Adjusting the constant $C$ of Lemma~\ref{lm-soficapproximationZwrZ} suitably, we have:

\begin{theorem}
\label{th-functions-ZwrZ}
For all $\ep >0$ and $k \geq 3$, there exists $C$ such that  if $n$ is coprime to $m$ and $|m| > C$ and $n > |m|^C$,
then there exists $f  \in \Sym(\Z /n\Z)$ which has order dividing $k$ and the property that $f(x+1) = m f(x)$ for at least $(1-\ep)n$ values of $x \in \Z /n\Z$.
\end{theorem}

 This result is stronger but less clean than the version we preferred to present  in the introduction as Theorem~\ref{heuristicapplication}.

\begin{proof}[Proof of Theorem~\ref{heuristicapplication}]
Theorem~\ref{heuristicapplication}  states that for all $\varepsilon >0$ and all $k\geq 3$, there exists $N \in \N$  such that for all coprime integers $m$ and $n$ with $n > N$ and  $\ln \ln n < m < \ln n$,
there exists an $f$ as per Theorem~\ref{th-functions-ZwrZ}.   This follows from Theorem~\ref{th-functions-ZwrZ} by taking $N$ sufficiently large that $\ln \ln N >C$ and
$N >   (\ln N)^C$.
\end{proof}

As mentioned above, $\H_k(\Z\wr \Z)$  is sofic for all $k \geq 4$ by Corollary~\ref{so sofic cor}.
For $k=1$ and $k=2$ the groups $\H_k(\Z\wr \Z)$ are sofic since
$$
\overline{\H}_1(\Z\wr \Z)  \ = \  \H_1(\Z\wr \Z)  \ \cong \  \Z
\quad \text{ and  } \quad
\overline{\H}_2(\Z\wr \Z)  \ \cong \  \Heis
$$
(but $\Z \wr \Z$ does not embed in $\H_1(\Z\wr \Z)$ or $\H_2(\Z\wr \Z)$).
As in the case of $\Heis$, we do not know whether $\H_3(\Z\wr \Z)$ is sofic, but we see no  reason  it should not be.

\begin{remark}
Before proving Theorem~\ref{ressolvable}, which implies that  $\H_4(\Z\wr\Z)$ is sofic, we constructed   finite quotients    which can be combined to give
a residually finite quotient $Q$ of $\H_4(\Z\wr\Z)$ in which $\Z \wr \Z$ embeds. We find these quotients  interesting on their own, and will briefly describe them.
Pick a prime $p$ and two functions $f, \lambda: \F_p \to \F_p^*$.
Then there is an action $\psi_{p,f,\lambda}$  of $\H_4(\Z\wr\Z)$ on $S = \F_p \by \F_p \by \F_p \by \F_p$ defined by
$$t : (x,y,z,w) \mapsto (y,z,w,x), \qquad  a: (x,y,z,w) \mapsto (x\lambda(z),y,z,w+f(z)).
$$
One computes that
$$
a^t:(x,y,z,w) \mapsto (x + f(w),y\l(w),z,w),
$$
and since $a^t = d$, we then can compute that
$$\begin{array}{ll}
a:(x,y,z,w) \mapsto (x\l(z),y,z,w+f(z)), & \ \   c:(x,y,z,w)  \mapsto (x,y+f(x),z\l(x),w), \\
b:(x,y,z,w)  \mapsto (x,y,z+f(y),w\l(y)),   & \ \  d:(x,y,z,w)  \mapsto (x+f(w),y\l(w),z,w).
\end{array}$$
To verify that this is an action, observe that $(a^t)^{a^j}$ acts via
$$
\left( a^t \right)^{a^j}:(x,y,z,w) \mapsto (x+\l(z)^{-j}f(w+jf(z)),y\l(w+jf(z)),z,w).
$$
Since this distorts $x$ and $y$ by only multiplication or only addition of elements that are not distorted (namely $z$ and $w$), $(a^t)^{a^j}$ will commute with $(a^t)^{a^i}$.
So, $\psi_{p,f,\lambda}$ is a homomorphism   $\H_4(\Z\wr\Z) \to \Sym(p^4)$.
Using a combinatorial description of the elements in $\Z\wr \Z$ it is possible to show that the restriction of  $\psi_{p,f,\lambda}$ to $\Z\wr \Z$ is injective.

These actions $\psi_{p,f,\lambda}$ are quite different from those arising in our proof of Theorem~\ref{ressolvable}.  We expect that for generic functions $f$ and $\lambda$, the image
of $\psi_{p,f,\lambda}$ will either be the full symmetric group, or the alternating group, and so will be very far from  (residually) solvable.   It is intriguing question whether the actions $\psi_{p,f,\lambda}$ distinguish all elements in $\H_4(\Z\wr \Z)$. We  see no  reason why this should not the case, but
without having an easily understandable combinatorial model of $\H_4(\Z\wr \Z)$, it is hard to prove such claim.
\end{remark}

\subsection{The free metabelian group on two generators}
\label{FMG section}

The free metabelian group on two generators $\Met$  has a presentation
$$
\Met = \left\langle a, b \  \left| \  \left[[a, b],[a,b]^{a^ib^j}\right] = 1 \ \forall i,j \in \Z \right. \right\rangle.
$$
It is a residually finite solvable group, so  is amenable.
For   $n \in \N$ and   $p$ and $q$  relatively prime to $n$, define a map $\psi_{n,p,q}: \Z \wr \Z \to \Sym(\Z/n\Z)$  by
\begin{align*}
\psi_{n,p,q}(a) &= \left(x \mapsto q^{-1}(x+1)\right) \text{, and} \\
\psi_{n,p,q}(b) &= \left(x \mapsto p^{-1}x \right),
\end{align*}
which extends to the  whole of $\Met$ since the permutations $\psi_{n,p,q}(a)$ and $\psi_{n,p,q}(b)$ satisfy the defining relations of $\Met$.
We have the analogue of Lemma~\ref{lm-soficapproximationZwrZ}, whose proof is practically the same:
\begin{lemma}
\label{lm-soficapproximationFMB}
For all finite sets $S \subseteq \Met$ and all $\d >0$, there exists a constant $C$ such that for all integers $n$, $p$, and $q$ with $n$ coprime to $p$ and $q$  and satisfying $n \nmid t(p,q)$
for all nonzero polynomials $t(x,y) = \sum t_{i,j} x^iy^j \in \Z[x,y]$ with integer coefficients $|t_{ij} | < C$ and total degree at most $C$, the map
$\psi_{n,p,q}$ is an $(S,\d,n)$-approximation of $\Met$.

If, moreover, $ |q| > 2C+1$, $|p| >  |q|^{C +1}$ and $n >  |p|^{ C +1}$, then all such polynomials satisfy  $n \nmid t(p,q)$, and so  $\psi_{n,p,q}$  necessary is an $(S,\d,n)$-approximation of $\Met$.
\end{lemma}

So Theorem~\ref{main-precise}  tells us that for all $k \in \N$ the group $\H_k(\Met)$ has a   sofic quotient into which $\Met$ naturally injects if and only there exist infinitely many
$n$ and functions $f: \Z/n\Z \to \Z/n\Z$ of order dividing $k$  which conjugate the actions of $\psi_{n,p,q}(a)$ and $\psi_{n,p,q}(b)$  up to error at most $\ep$.

For $k\geq 4$ we can use Corollary~\ref{so sofic cor} to see that $\H_k(\Met)$ is sofic and  Lemma~\ref{injective} to see that $\Met$ naturally embeds into it.  The case $k=2$  is handled by the observation
$\overline{\H}_2(\Met) \cong \Met$, and $k=3$ by Proposition~\ref{soficquotient}. (We do not know whether $\H_3(\Met)$ is sofic.)  As for the case $k=1$, the group $\H_1(\Met)$ is also  sofic since $\overline{\H}_1(\Met)  = \H_1(\Met) \cong \Z$ but $\Met$ does not embed in  $\H_1(\Met)$.

Theorem~\ref{main-precise} then allows us to conclude:

\begin{theorem}
\label{th-functions-FMB}
For all $\ep >0$ and $k \geq 2$, there exists $C$ such that if $n$ is coprime to $p$ and $q$ and $|q| > C$, $|p| > |q|^C$ and $ n > |p|^C$, then there exists $f \in \Sym(\Z /n\Z)$
such that $f^k = \id$ and $f(qx+1) = p f(x)$ for at least $(1-\ep)n$ values of $x \in \Z /n\Z$.
\end{theorem}

\section{Heuristic} \label{heuristics}
Here we will explain why the existence of the permutations $f \in \Sym(n)$ proved in Theorem~\ref{main-precise}  is surprising.
We will focus on instances where $A=\langle a \rangle \cong B = \langle b \rangle \cong \Z$ and $k=4$  which is the case in most of our examples.
(We could generalize to $k \geq 4$ without significantly changing the following argument, but the assumption that   $A \cong B \cong \Z$   is essential.)

Denote $\alpha = \psi(a)$ and  $\beta = \psi(b)$. Each permutation  $f$ satisfies the global condition $f^4 =\id$; and many local conditions: the  condition concerning
$d( \psi (b) \circ f, f \circ \psi(\phi(b)))$ is equivalent to
$
f (\alpha (x)) = \beta (f (x))
$
for at least $(1-\ep) n$ points $x$.

One can estimate  the probability that a permutation $f$ chosen uniformly at random from $\Sym(n)$ satisfies that global condition:
by considering cycle structure, one counts the number of elements of order dividing $4$ in the symmetric group $\Sym(n)$ (see~\cite{chs}), which leads to
$$
P  \ = \  \mathrm{Prob}( f^4 = \id)  \ \approx \  \frac{1}{\sqrt[4]{|\Sym(n)|}}  \ \approx  \ n^{-n/4}.
$$
(Here we are only describing the leading term of the expansion of $\log P$.)

It is also quite easy to estimate the probability that a local condition is satisfied.
A local condition asserts that $f(\alpha(x))$ is determined by
$f(x)$---the probability that this happens at a given $x$  is approximately $1/n$ (only approximately since $x$ might be a fixed point of $\alpha$ or $f(x)$ might be a fixed point for $\beta$).
However, this probability is irrelevant since we want
$f (\alpha (x)) = \beta (f (x))$
for the majority of $x$ (for at least $(1-\ep) n$ points $x$, to be precise).
Informally, a small number of local conditions are almost independent from each  other, but this is not true if we consider many local conditions.

The number of permutations in $\Sym(n)$ satisfying $f (\alpha (x)) = \beta (f (x))$  for at least $(1-\ep) n$ points $x$ is at most $n^{2 \ep n +k}$ where $k$ is the number cycles in the action of $\alpha$---such
permutations are determined by the following: the points $x$ where the local condition is not satisfied; values of $f(\alpha(x))$ at each of these points; and the values of $f$ at a single point on each cycle on the action of $\alpha$. (This  information specifies a function $f$ but it may not be a permutation.) For all $\ep' >0$, we have
$k < \ep' n$ for large $n$, otherwise a small power of $\alpha$ will be close to the identity permutation, which would contradict the fact that $\psi$ detects the soficity of $G$.
Thus, the number of permutations satisfying the majority of the local conditions is at most $ K = n^{(2 \ep + \ep') n}$.

If the global condition is almost independent  from the local conditions, then the expected number of permutation satisfying both, should be around
$$
P . K  \ = \   n^{-n/4}. n^{(2 \ep + \ep') n}  \ = \  n^{(2 \ep + \ep' -1/4 ) n} \ll 1,
$$
if $\ep, \ep' < 1/20$. Thus, one should expect  that there are no such permutations when $n$ is sufficiently large.

The independence  assumption is somewhat justified by the observation that the global condition is independent form each of the local conditions (it is also almost independent from any fixed number of local conditions). Notice that this heuristic does not really depend on the group $G$.

The main weakness of this heuristic is the assumption that the global condition  is almost independent from the majority of the local conditions.
One can interpret Theorem~\ref{main-precise} as saying there is a connection between the soficity
of the group $\H_k(G)$ and the independence of  global versus local conditions.

\section*{Acknowledgements}

We thank Harald Helfgott  for discussions on this project and are grateful for partial support from NSF grants DMS 1303117 and 1601406 (Kassabov),  the Rawlings College Presidential Research Scholars program at Cornell (Kuperberg), and   Simons Collaboration Grant 318301 (Riley).

\bibliographystyle{amsplain}
\bibliography{hwreath}

\bigskip

   \href{https://www.math.cornell.edu/m/People/bynetid/mdk35}{\textsc{Martin Kassabov}}, \texttt{kassabov@math.cornell.edu} \\
  \href{http://www.math.cornell.edu/~viviankuperberg/}{\textsc{Vivian Kuperberg}},
  \texttt{viviank@stanford.edu} \\
  \href{http://www.math.cornell.edu/~riley/}{\textsc{Timothy Riley}},   \texttt{tim.riley@math.cornell.edu} \\[2mm]
MK and TR: Dept.\ of Mathematics, 310 Malott Hall,  Cornell, Ithaca, NY 14853, USA \\
VK: Dept.\ of Mathematics, 450 Serra Mall, Stanford University, Stanford, CA 94305, USA

\end{document}